\documentclass[12pt]{article}
\usepackage{latexsym}
\usepackage{amssymb}
\usepackage{amsmath}
\usepackage{rotating}
\usepackage{multirow}
\usepackage{mathrsfs}
\usepackage{wasysym}
\usepackage{esint}
\usepackage{rawfonts}
\input{prepictex}
\input{pictex}
\input{postpictex}
\usepackage[OT2,OT1]{fontenc}
\def\cyr{%
\renewcommand\rmdefault{wncyr}%
\renewcommand\sfdefault{wncyss}%
\renewcommand\encodingdefault{OT2}%
\normalfont
\selectfont}
\DeclareMathAlphabet{\zap}{OT1}{pzc}{m}{it}
\DeclareTextFontCommand{\textcyr}{\cyr}
\def\be{\begin{equation}}
\def\ee{\end{equation}}
\def\bea{\begin{eqnarray*}}
\def\eea{\end{eqnarray*}}

\def\Bbb{\mathbb}

\def\CC{\mathbb C}

\newtheorem{main}{Theorem}

\DeclareMathOperator{\Hess}{Hess}

\newtheorem{thm}{Theorem}
\newtheorem{lem}{Lemma}[section]
\newtheorem{prop}{Proposition}

\newenvironment{proof}{\medskip \noindent
{\bf Proof.}}{\hfill \rule{.5em}{1em}
\\}

\def\ZZ{{\mathbb Z}}
\def\RR{{\mathbb R}}
\def\CP{{\mathbb C \mathbb P}}
\begin{document}

\title{Einstein Manifolds and\\ Extremal K\"ahler Metrics}

\author{Claude LeBrun\thanks{Supported 
in part by  NSF grant DMS-0905159.}\\SUNY Stony
 Brook}

\date{}
\maketitle

 \begin{abstract}	
 In joint work with Chen and Weber \cite{chenlebweb} ,  the author has elsewhere 
 shown that $\CP_2 \# 2\overline{\CP}_2$ admits an Einstein metric. 
  The present  paper gives   a new and rather different proof of this fact. 
Our results include  new existence theorems for
 extremal K\"ahler metrics, and these allow one  to prove
 the above existence statement by deforming 
 the K\"ahler-Einstein metric on $\CP_2 \# 3\overline{\CP}_2$
 until  bubbling-off occurs. 
 \end{abstract}

\section{Introduction}

Recall that a Riemannian metric is said to be {\em Einstein}
iff it has constant Ricci curvature  \cite{bes}. When this happens, the 
constant value assumed by the Ricci curvature 
is called the Einstein constant.  
A fundamental  problem of global  
Riemannian geometry is to determine  precisely  which smooth 
compact manifolds admit Einstein  metrics.

While we still remain quite far from being able to determine precisely which 
smooth compact $4$-manifolds admit Einstein metrics, notable
 progress
has  recently been achieved regarding  
  narrower versions of the problem. For example, it was shown
in \cite{chenlebweb} that if $M^4$ is  the underlying smooth oriented manifold of 
a compact complex surface, then $M$ admits an Einstein metric 
{with positive Einstein constant} iff  it  is diffeomorphic to a Del Pezzo surface. In other words, 
the only allowed diffeotypes are $S^2\times S^2$ and $\CP_2 \# k \overline{\CP}_2$, 
with   $0\leq k \leq 8$;  and, conversely, each  of these candidates does 
 actually admit   an Einstein metric with positive Einstein constant.
 
 The existence direction  of the above assertion    is proved by means  of 
 K\"ahler geometry. For example, the  theory of the complex Monge-Amp\`ere equation can
 be used to show  that all but two of the above  candidates admit 
 K\"ahler-Einstein metrics  \cite{s,tian,ty}.  
However,  $\CP_2 \#  \overline{\CP}_2$ and $\CP_2 \# 2 \overline{\CP}_2$
 cannot admit K\"ahler-Einstein metrics, owing to 
 the non-reductive nature of their automorphism groups \cite{mats}. 
  Nonetheless, Page \cite{page} was able to construct an explicit 
  cohomogeneity-one 
  Einstein metric on $\CP_2 \# \overline{\CP}_2$, and Derdzi{\'n}ski
  \cite{derd} subsequently discovered that this metric is actually  {\em conformally}
  K\"ahler. 
Following this lead, Chen, Weber, and the present author  later proved 
  \cite{chenlebweb} 
that $\CP_2 \# 2\overline{\CP}_2$  also admits 
a conformally K\"ahler, Einstein metric, although, in contrast to  the 
the Page metric, we do not know its  explicit form. 

The present article will  provide a  new proof for the existence of an Einstein metric
on  $\CP_2 \# 2\overline{\CP}_2$. In the process, we 
will actually give a direct proof of the following, ostensibly  stronger statement: 

\begin{main} \label{pace} 
There is a  conformally  
K\"ahler, Einstein metric
$h$ on $M=\CP_2\# 2 \overline{\CP}_2$
for which the  conformally related K\"ahler metric $g$ 
minimizes the $L^2$-norm of the scalar curvature among
all K\"ahler metrics on $M$. In other words, $h$ is an absolute minimizer of the 
Weyl functional among all conformally K\"ahler metrics on $M$.
\end{main} 

For a definition of the Weyl functional, see \S \ref{bach} below. Notice that 
Theorem~\ref{pace} does not  actually  assert that the relevant Einstein metric 
actually coincides with the metric of \cite{chenlebweb}. However, 
  this is  actually true, 
  as a consequence  of   uniqueness results recently 
  proved in  \cite{lebuniq}. 

Another  main purpose of the present article is to  prove new existence results 
 for extremal K\"ahler metrics. On any toric Del Pezzo surface, we 
 show, in  Theorems \ref{laudate} and \ref{gaudete} below, 
 that any K\"ahler class in a  large, specific  neighborhood of $c_1$ is represented by an 
 extremal K\"ahler metric.
  This    not  only   implies  Theorem \ref{pace}, 
  but also allows us to prove the following:

\begin{main}\label{quoniam}
Let $g$ be the conformally Einstein, 
 K\"ahler metric on $\CP_2 \# 2\overline{\CP}_2$ 
 discussed in Theorem \ref{pace}. Then there is a 
$1$-parameter family $g_t$, $t\in [0,1)$ of extremal K\"ahler metrics 
 on  $\CP_2 \# 3\overline{\CP}_2$,
such that $g_0$ is K\"ahler-Einstein, and such that $g_{t_j}\to g$
in the Gromov-Hausdorff sense for some $t_j\nearrow 1$. 
\end{main}

The approach that will be developed here hinges on a systematic   study of the 
squared $L^2$-norm 
$${\mathcal C}(g)= \int_M s^2_g ~d\mu_g$$
of the scalar curvature, restricted to the space of K\"ahler metrics. 
An important  variational problem for this functional was first
studied by Calabi \cite{calabix,calabix2}, 
who constrained  $g$ to only vary in a  fixed  K\"ahler class
$[\omega ] \in H^{2} (M,\RR)$.  Calabi called the  critical metrics of
his restricted problem  {\em extremal K\"ahler metrics}, and showed
that the relevant Euler-Lagrange equations 
 are equivalent to requiring that
$\nabla^{1,0}s$ be a holomorphic vector field. In fact, 
every extremal K\"ahler metric turns out to be an 
absolute minimizer for the Calabi problem, and the proof of this  \cite{xxel}
moreover implies  the sharp estimate 
$$
\frac{1}{32\pi^2}\int_M s^2_g ~d\mu_g \geq 
\frac{(c_1\cdot \Omega )^2}{\Omega^2}+ \frac{1}{32\pi^2}\|{\mathfrak F}(\Omega) \|^2
$$
with equality iff $g$ is an extremal K\"ahler metric. Here 
$${\mathfrak F}(\Omega ) : H^0(M, {\mathcal O}(T^{1,0}M))\to \CC$$
denotes the Futaki invariant, and the relevant norm is the one induced by the 
$L^2$-norm on the space of  holomorphy potentials  \cite{fuma0}. In particular, 
for any extremal K\"ahler metric $g$ with K\"ahler class $\Omega$, one has 
\begin{eqnarray*}
\int_M s_0^2 ~d\mu_g &=& 32\pi^2 \frac{(c_1\cdot \Omega )^2}{\Omega^2}\\
\int_M (s-s_0)^2_g ~d\mu_g &=& \|{\mathfrak F}(\Omega)\|^2
\end{eqnarray*}
where 
$$
s_0= \fint_M s ~d\mu_g = \frac{\int s~d\mu_g}{\int d\mu_g}
$$
denotes the average scalar curvature. 
Thus, letting ${\zap K} \subset H^2 (M, \RR)$ be the K\"ahler cone of 
$(M,J)$, we are led to consider  the action function
  ${\mathcal A}: {\zap K} \to \RR$
  defined by 
  \begin{equation}
\label{action}
{\mathcal A}(\Omega) =  
\frac{(c_1\cdot \Omega )^2}{\Omega^2}+ \frac{1}{32\pi^2}\|{\mathfrak F}\|^2
\end{equation}
which we  write  schematically as
$$
{\mathcal A}(\Omega) = {\mathcal T} (\Omega ) + {\mathcal B} (\Omega ) 
$$ 
where 
$$
{\mathcal T} (\Omega )  = \frac{(c_1\cdot \Omega )^2}{\Omega^2}
$$
is a manifestly   {\em topological term}, and where 
the Futaki term 
$$
 {\mathcal B} (\Omega ) = \frac{1}{32\pi^2}\|{\mathfrak F}(\Omega )\|^2
$$
will eventually be shown to  be uniformly {\em  bounded}.

In practice, ${\mathcal A}$ is an explicitly computable
 rational function of several variables. Nonetheless, the 
actual expression is sufficiently  complicated  that 
the judicious use of  computer algebra is of enormous help in 
reliably obtaining the correct answer. For this reason, 
a number of the proofs presented here partially depend  on  
calculations carried out with the assistance of  {\em Mathematica}.
However, 
these calculations are merely elaborate 
 algebraic  manipulations which 
could, in principle,  be directly verified by  a careful human with 
sufficient time and patience. 
For clarity of presentation, we have grouped  these computer-assisted
  calculations  into 
two appendices. In spite of their typographical location, these appendices are
logically independent on the rest of the paper, and  might 
rightly be said to represent  the beginning rather than the end of the  article.
For this reason, results proved in the appendices  are  freely cited 
throughout the body of the paper, without the slightest danger  of circular reasoning.

\section{Extremal K\"ahler Metrics} \label{xtr} 

In this section, we will prove two results on the existence of extremal K\"ahler metrics.
While these appear  \cite{dontor2}  to      be   of genuinely  independent  interest, 
we will be principally interested in them here because of the 
key  role they will play  in the proofs 
of Theorems \ref{pace} and \ref{quoniam}. 

\begin{thm} \label{laudate}
Let $M\approx \CP_2\# 2 \overline{\CP}_2$ be the blow-up of 
$\CP_2$ at two distinct points, and let $[\omega ]$ be a K\"ahler class
 on $M$ for which 
 $${\mathcal T}([\omega]) := \frac{(c_1\cdot [\omega])^2}{[\omega ]^2} \leq 
 \frac{3}{2} c_1^2 - \frac{1}{4} =  c_1^2 + 3.25. $$ 
 Then there is an extremal K\"ahler metric $g$ on $M$ with K\"ahler form 
 $\omega \in [\omega ]$. 
\end{thm}

\begin{thm}\label{gaudete} 
Let $M\approx \CP_2\# 3 \overline{\CP}_2$ be the blow-up of 
$\CP_2$ at three non-collinear points, and let $[\omega ]$ be a K\"ahler class
 on $M$ for which 
 $${\mathcal T}([\omega]) := \frac{(c_1\cdot [\omega])^2}{[\omega ]^2} \leq
   \frac{3}{2} c_1^2 - \frac{1}{4} =  c_1^2 + 2.75.$$ 
 Then there is an extremal K\"ahler metric $g$ on $M$ with K\"ahler form 
 $\omega \in [\omega ]$. 
\end{thm}

To prove these results, we will rely on a continuity method argument analogous  to the
one
 used to prove 
\cite[Theorem 27]{chenlebweb}. This time, however, 
we will start at the anti-canonical
class $c_1$ and work outward. To make this feasible, we will 
temporarily just   assume that $c_1$ is represented by 
an extremal K\"ahler metric. For  $M= \CP_2\# 3 \overline{\CP}_2$, 
this is assumption is certainly valid, since the anti-canonical class class contains the 
K\"ahler-Einstein metric of Siu \cite{s}.  In the case of  $M= \CP_2\# 2 \overline{\CP}_2$,
we will eventually validate this assumption  by providing a new 
proof  in  Proposition \ref{sesame} below; in the interim, however,
we will try to put the 
  reader at ease by mentioning  that this fact has  actually  been previously 
 proved elsewhere \cite{chenlebweb,he} using   different methods.

By rescaling, we may also assume that 
the  target K\"ahler class $[\omega ]$ satisfies 
$c_1\cdot [\omega ] = c_1^2$, so   that $[\omega ] = c_1+\eta$ for some
$\eta \in H^2(M, \RR)$ with $c_1\cdot \eta =0$. We now join 
 $c_1$ to the given 
$[\omega]$ by a straight line segment
$$[0,1]\ni t\longmapsto [\omega_t]:= (1-t) c_1+ t [\omega] = c_1+ t\eta$$
and notice that the $[\omega_t]$ are all K\"ahler classes, by convexity of the K\"ahler cone. 
Since $\eta^2 < 0$, 
$$[\omega_t]^2 = (c_1+t\eta)^2 = c_1^2 +t^2 \eta^2 \geq c_1^2 + \eta^2 = [\omega]^2,$$
so that 
$${\mathcal T}([\omega_t]) = \frac{(c_1\cdot [\omega_t])^2}{[\omega_t]^2} \leq 
\frac{(c_1\cdot [\omega])^2}{[\omega]^2} = {\mathcal T}([\omega]) \leq  \frac{3}{2} c_1^2 - \frac{1}{4} ~~\forall t\in [0,1].$$
We are  therefore required to prove the existence of a solution in each $[\omega_t]$. It is therefore natural 
to  consider the set 
$${\zap E} = \left\{ t\in [0,1]~|~ [\omega_t] \mbox{ is represented by an extremal K\"ahler metric}
\right\},$$
and define 
$${\mathfrak t} = \sup \{ t\in [0, 1]~|~ [0,t]\subset {\zap E}\}.$$
We have already assumed   that $0\in {\zap E}$, so ${\zap E}\neq \varnothing$,
and ${\mathfrak t}\in [0,1]$. 
On the other hand,  an inverse-function theorem  argument \cite{ls2}
implies that ${\zap E}$ is open in $[0,1]$. One connected component
of   ${\zap E}$  therefore   either takes the form $[0, {\mathfrak t})$ or $[0,1]$.
It therefore suffices to show 
that ${\mathfrak t}\in {\zap E}$, as this will then immediately  imply that ${\zap E}= [0,1]$.

To attain this goal, we will make systematic use of  the weak compactness theorem of 
Chen and Weber \cite{chenweb}. This result guarantees that, given a sequence
of unit-volume extremal K\"ahler metrics on a compact complex surface, 
one can extract a subsequence
which Gromov-Hausdorff converges to an extremal K\"ahler orbifold metric, 
{\em provided}  there is a uniform upper bound on the Sobolev constants.
Such an upper bound can in turn be guaranteed \cite{chenlebweb} if the metrics in question 
have uniformly  bounded, positive scalar curvature, 
and if all belong to the controlled cone 
\begin{equation}
\label{control}
{\mathcal A} ([\omega ]) < \frac{3}{2} c_1^2 -\epsilon
\end{equation}
for some $\epsilon > 0$, where 
$$\mathcal A = \frac{1}{32\pi^2} \int_Ms^2d\mu$$
for an extremal K\"ahler metric. However,  
$${\mathcal A} = {\mathcal T} + {\mathcal B}, $$
where
$$
\mathcal B = \frac{1}{32\pi^2} \int_M(s-s_0)^2d\mu
$$
for an extremal K\"ahler metric, and 
we  show in Lemmas \ref{up1} and \ref{up2}  that 
$$
\mathcal B < \frac{1}{4}
$$
for every K\"ahler class on either of these manifolds. On either of these manifolds, 
it follows that  the inequality (\ref{control}) 
holds for any convergent sequence of K\"ahler classes with 
$$
{\mathcal T} \leq \frac{3}{2}c_1^2 -\frac{1}{4}
$$
where $\epsilon$ is the infimum  of $\frac{1}{4} - {\mathcal B}$
over a small neighborhood of the limit class.
Moreover, Lemmas \ref{pos2} and \ref{pos3}  show
that  extremal K\"ahler metrics on these manifolds always have 
 everywhere-positive scalar curvature which is  uniformly bounded
on any compact  subset of the K\"ahler cone ${\zap K}$. 
 Hence, by rescaling to unit volume and then scaling back, 
every sequence in ${\zap E}$ 
has a subsequence for which the corresponding extremal K\"ahler manifolds $\{ (M, g_{j})\} $  
converge to an extremal K\"ahler orbifold $(N, g_\infty)$ in the Gromov-Hausdorff topology. 
Of course,   the K\"ahler classes $[\omega_j]$  may  simultaneously be taken to 
converge to some K\"ahler class $\Omega
\in \{ [\omega_t]~|~t\in (0,1]\}$.  

We will now  specialize to the case of an 
increasing sequence $t_j \nearrow {\mathfrak t}$, with the goal of showing that 
${\mathfrak t}\in {\zap E}$. 
In order  to show that, modulo diffeomorphisms, the $g_{j}$ actually converge  smoothly  
to a  metric
on the given  $M$, we must  rule out bubbling. Recall \cite{chenlebweb,chenweb}
that smooth convergence will fail only if the sectional curvatures of
our metrics $g_j$ fail to be uniformly bounded, and that when this
happens, after once again passing to a subsequence,  one can 
 find a sequence of rescalings  $\kappa_j^{-1}g_j$, $\kappa_j \to 0$,  and
a sequence of base-points  $p_j\in M$ 
such that $\{ (M,p_j , \kappa_j^{-1}g_j )\}$ 
converges in the pointed Gromov-Hausdorff topology  to a non-trivial  
ALE scalar-flat K\"ahler surface $(X,\hat{g}_\infty)$. Such a pointed limit is
 called a {\em deepest bubble}. 
Because all the metrics in our sequence are toric, so is the 
deepest bubble. This implies \cite[Lemma 17]{chenlebweb}
that $b_2(X)\neq 0$, and that $H_2(X,\ZZ)$ is generated by embedded
holomorphic $\CP_1$'s. Moreover, for large $j$ in the subsequence, 
the pointed Gromov-Hausdorff convergence guarantees that 
$X$ is diffeomorphic
to an  open subset $U_j\subset M$, in such a manner that 
$c_1(X)$ 
obtained by restricting $c_1(M)$ to $U$, and such that $H_2(U_j)$ is generated
by embedded $2$-spheres $S_j$ which are symplectic with respect to the K\"ahler
form $\omega_j$. Finally, the 
homomorphism $H_2(U_j,\ZZ)\to H_2(M,\ZZ)$ induced by inclusion is injective, and
the restriction of the intersection form of $M$ to $U_j$ is negative definite. 

Our strategy will  now  combine ideas  from \cite{chenlebweb} and \cite{bing}. 
Suppose that $(X,\hat{g}_\infty)$ is a deepest bubble arising from some 
sequence $g_j:=g_{t_j}$, where $t_j\nearrow {\mathfrak t}$. Let $S\subset X$
be any holomorphic embedded $\CP_1$, and let $k> 0$ be
the positive integer defined  by $S\cdot S=-k$. 
Then for each  $j$ sufficiently
far out in our subsequence, we can find an $\omega_{t_j}$-symplectic 
$2$-sphere $S_j\subset M$ with $S_j\cdot S_j = -k$, for some fixed positive integer 
$k$. By the adjunction formula, we  would then also have 
 $c_1\cdot S_j = 2-k$.  
As $j$ varies, 
the homology class $[S_j]$ could  in principle change. However, 
since $c_1^2 > 0$ and  $b_+(M)=1$, 
the subset of $H_2(M, \RR)$ defined by
\begin{eqnarray*}
c_1\cdot A &=& 2-k\\
A\cdot A &=& -k
\end{eqnarray*}
is compact, and the set of  $A\in H_2(M, \ZZ)$ satisfying these conditions
is therefore finite. By refining our subsequence, we may therefore assume that 
$[S_j]=[S]$ is independent of $j$. Moreover, since $S$ has finite area
in $(X,\hat{g}_\infty)$, which is a rescaled limit of regions $U_j\subset M$, with magnification 
tending to infinity, we must be able to represent $[S]$ by symplectic $2$-spheres
$S_j$ of arbitrarily small area in $(M, g_j)$ as $j\to\infty$. Since
the area of $S_j$ is $ \geq |[\omega_j]\cdot  [S_j] |$ by Wirtinger's inequality, 
taking the limit as $j\to \infty$  now yields 
$$\Omega \cdot A = 0,$$
where $\Omega = [\omega_{\mathfrak t}]$ is the limit K\"ahler class. 

On the other hand, the sphere $S_j$ is symplectic with respect to
each K\"ahler form $\omega_j= \omega_{t_j}$  in our subsequence. Now, by construction, 
$$[\omega_{j}] =  u_j c_1 + (1-u_j ) \Omega$$
for a sequence of positive numbers $u_j = 1-(t_j/{\mathfrak t})\searrow 0$. Since 
$S_j$ is symplectic, we therefore have 
$[\omega_{j}]\cdot A = [\omega_{j}]\cdot [S] > 0$, for large $j$. Since 
$\Omega \cdot A=0$, this says  that $u_j (c_1 \cdot A  ) = [u_j c_1 + (1-u_j)\Omega ] \cdot A> 0$.
Hence 
$$c_1\cdot A > 0.$$
However,
$c_1\cdot A = 2-k$ by the adjunction formula. We thus conclude that  $k < 2$. 
It follows that $k=1$, thereby reducing  our bubbling problem 
to a single  case.

To deal with the  remaining $k=1$ case, we  now 
 classify the homology classes $A\in H_2(M, \ZZ)$ satisfying 
\begin{eqnarray*}
c_1\cdot A &=& 1 \\
 A \cdot A &=& -1~.
\end{eqnarray*}
 For this purpose, it is best to concentrate on the case of $M=\CP_2 \# 3 \overline{\CP}_2$,
since we can identify $H^2(\CP_2 \# 2 \overline{\CP}_2)$
with a hyperplane in $H^2(\CP_2 \# 3 \overline{\CP}_2)$.
 If we choose a basis for $H_2(\CP_2 \# 3 \overline{\CP}_2)$
consisting of a projective line $L$ and three exceptional divisors 
$E_1$, $E_2$, $E_3$, corresponding to the three blown-up points in
$\CP_2$, the intersection form  then
becomes 
$$
\left(\begin{array}{rrrr}1 &  &&\\ & -1&&\\
 & & -1&\\
 & & &-1\\\end{array}\right)
$$
and $c_1$ is Poincar\'e dual to $(3,-1,-1,-1)$.  
Setting $A=(n,a,b,c)$, we thus have 
\begin{eqnarray*}
3n+a+b+c&=&1\\
n^2-a^2-b^2-c^2 &=& -1
\end{eqnarray*}
and it therefore follows that 
$$5(a^2+b^2+c^2)+(a-b)^2+(a-c)^2+(b-c)^2+(a+1)^2+(b+1)^2+(c+1)^2=13.$$
In particular, $a^2+b^2+c^2 < 3$. On the other hand, 
 $a^2+b^2+c^2= n^2 + 1\geq 1$. Thus, after possibly permuting $E_1$, $E_2$, 
 and $E_3$, we may assume that $c=0$,  that $|a|= 1$, and that $|b|\leq 1$.  
 It is then easy to check that, again modulo permutations of the 
$E_j$, the only solutions are
 $$
 (n,a,b,c)= (0,1,0,0) \mbox{ and } (1,-1,-1,0),
 $$
 respectively corresponding to 
 $$A= E_1  \mbox{ and } A=L-E_1-E_2.$$
 Throwing in permutations, we conclude that there are
 exactly six possibilities 
 $$A= E_1, ~E_2, ~E_3, ~L-E_1-E_2, ~L- E_1 - E_3, ~L-E_2 - E_3$$
 on $M=\CP_2 \# 3 \overline{\CP}_2$, and exactly three possibilities 
  $$A= E_1, ~E_2,  ~L-E_1-E_2$$
  on $M=\CP_2 \# 2 \overline{\CP}_2$. 
 This is good news, because  these classes are actually all  represented by 
 holomorphic (-1)-curves on either choice of 
 $M$. Since these holomorphic  curves must have positive area
 for any K\"ahler metric, any K\"ahler class $\Omega$ on $M$ must therefore satisfy
 $\Omega \cdot A > 0$ for any such class $A=[S]$. This rules out  bubbling 
 when $k=1$, and our previous argument therefore shows that bubbling  has now been  
 definitively
 ruled out
 in 
 all cases.

As $t_j \to {\mathfrak t}$, 
 the sectional curvatures  of the $g_{j}$ therefore 
 remain uniformly bounded, and these metrics 
therefore converge to a smooth, toric, extremal K\"ahler metric
on a complex surface diffeomorphic to $M$.
 The collection of totally geodesic holomorphic curves 
consisting of  the points of non-trivial isotopy must converge to a configuration
of  totally geodesic holomorphic curves  with the same self-intersections as
the original curves in $M$, and with areas obtained by taking na\"{\i}ve limits,
allowing us to read off both the limit complex structure and the limit K\"ahler class. 
This shows that the limit extremal K\"ahler metric is actually compatible
with the original complex structure on $M$, with K\"aher class $\Omega$.
Thus ${\mathfrak t}\in  {\zap E}$, and  hence ${\zap E}= [0,1]$.
The target K\"ahler class  $[\omega ]$ therefore contains an extremal K\"ahler metric,
and Theorems \ref{laudate} and \ref{gaudete} have therefore been proved,
modulo the assumption that $c_1$ contains an extremal K\"ahler metric. 
We now complete the argument by justifying this last   assumption.  

\begin{prop} \label{sesame}
The anti-canonical class  of $\CP_2 \# 2 \overline{\CP}_2$ is represented
by an extremal K\"ahler metric. 
\end{prop}
\begin{proof}
We begin by once 
 again  recalling   \cite{s,ty}  that $c_1$ is 
represented on
$\CP_2 \# 3 \overline{\CP}_2$ by  a K\"ahler-Einstein metric; thus, 
 we can  safely apply the above arguments to $\CP_2 \# 3 \overline{\CP}_2$
 without assuming anything about $\CP_2 \# 2 \overline{\CP}_2$.
 We now identify $\CP_2 \# 2 \overline{\CP}_2$ with the blow-down of
 of $\CP_2 \# 3 \overline{\CP}_2$ along the exceptional divisor $E_1$,
 and let 
 $${\mathbf p} : \CP_2 \# 3 \overline{\CP}_2 \to \CP_2 \# 2 \overline{\CP}_2$$
 denote the blowing-down map. If $\Omega$ is any K\"ahler class on $ \CP_2 \# 2 \overline{\CP}_2$,
 then
 $$[\omega_t] = (1-t) c_1 + t ({\mathbf p}^* \Omega)$$
 is  a K\"ahler class on $M=\CP_2 \# 3 \overline{\CP}_2$ for any
 $t\in [0,1)$, and this K\"ahler class satisfies ${\mathcal T} ([\omega_t ])\leq {\mathcal T} (\Omega )$
 for any $t$, since the Poincar\'e dual of  $c_1 (\CP_2 \# 2 \overline{\CP}_2)$ 
 is exactly the push-forward,  via $\mathbf p$, of the 
 Poincar\'e dual of $c_1  (\CP_2 \# 3 \overline{\CP}_2)$.  
We now assume henceforth  that  ${\mathcal T} (\Omega)< 8.75$;
 this holds in particular, 
  if we  take $\Omega = c_1 (\CP_2 \# 2 \overline{\CP}_2)$,
 but we shall also allow for other possibilities,  as doing so will cost us no additional effort.
 
 Since ${\mathcal T} ([\omega_t ])\leq {\mathcal T} (\Omega ) < 8.75$, there is, 
 by  Theorem \ref{gaudete}, 
an  extremal K\"ahler metric $g_t$ on $M$  with K\"ahler form $\omega_t\in [\omega_t]$
for all 
$t\in [0,1)$. Moreover,  
 $[\omega_t]$ satisfies (\ref{control}) for all $t\in [0,1)$, with  
 $\epsilon = 8.75 - {\mathcal T} (\Omega )$. 
Because the  $g_t$ also have positive, uniformly bounded 
scalar curvatures by Lemma \ref{pos3},   these metrics therefore have uniformly bounded
 Sobolev constants, and the Chen-Weber theorem therefore guarantees 
 the existence of a  Gromov-Hausdorff limit  of some sequence $g_{t_j}$, 
 $t_j\nearrow 1$, with  limit  a compact 
extremal K\"ahler orbifold $(N,J_\infty , g_\infty )$. On the other hand, the sectional curvatures
of the $g_t$ are 
are certainly  {\em not} uniformly
bounded as $t\to 1$, as the presence of 
a totally geodesic $2$-sphere of area $\alpha = 1-t\searrow 0$, 
forces    $\sup K\nearrow +\infty$  by the classical Gauss-Bonnet theorem. 
Thus,   
a non-trivial deepest bubble must  arise.
On the other hand, our symplectic argument to
rule out the bubbles containing a spherical class $A$ with $A^2 = -k$
still works for $k\geq 2$; the only difference  is that the limit class
${\mathbf p}^*\Omega\in H^2(M, \RR)$  no longer belongs to the K\"ahler cone, 
but rather sits on its boundary. We thus  
 have still  excluded any 
deepest bubble except one  whose
homology is carried by    $(-1)$-curves. Moreover, $E_1$
 is the only 
homological $(-1)$-curve 
whose symplectic area tends to zero as $t\to 1$.
Since $E_1$ has 
 non-zero self-intersection, it cannot be simultaneously
arise from two disjoint bubbles,  or  from 
 a bubble on a bubble.  Thus, the limit
orbifold must obtained from $M=\CP_2 \# 3 \overline{\CP}_2$ by collapsing
a single $2$-sphere representing $E_1$. Since the link of such a
$2$-sphere is simply connected, it follows that the $N$ is a manifold, with $b_2=3$; 
similarly,  the bubble has $b_2=1$, and  is asymptotically Euclidean, rather than merely
being ALE. Consequently, the bubble that forms must be  the Burns metric \cite{lpa} on 
${\mathcal O}(-1)$ line bundle
over $\CP_1$,
since \cite{burns,poonthesis}, up to homothety,   
this is the only asymptotically Euclidean scalar-flat K\"ahler surface
with $b_2=1$. 

Because the Burns metric has isometry
group $U(2)$, 
the toric structure of the bubble is therefore  uniquely determined up to conjugation, with 
points of non-trival isotropy 
given by  the zero section and two fibers of the line bundle. 
By contrast, the set of points of  $M=\CP_2 \# 3 \overline{\CP}_2$
at which the torus action has non-trivial isotropy consists of the six $(-1)$-curves 
 
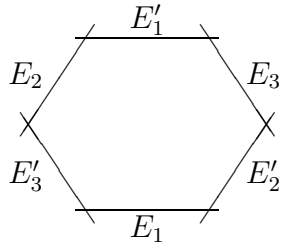
\begin{center}
\begin{picture}(100,80)(0,3)
\put(-7,70){\line(1,0){54}}
\put(40,75){\line(2,-3){28}}
\put(0,75){\line(-2,-3){28}}
\put(-7,5){\line(1,0){54}}
\put(40,0){\line(2,3){28}}
\put(0,0){\line(-2,3){28}}
\put(20,-2){\makebox(0,0){$E_1$}}
\put(20,77){\makebox(0,0){$E_1^\prime$}}
\put(-26,56){\makebox(0,0){$E_2$}}
\put(64,56){\makebox(0,0){$E_3$}}
\put(-26,18){\makebox(0,0){$E_3^\prime$}}
\put(64,18){\makebox(0,0){$E_2^\prime$}}
\end{picture}
\end{center}
where $ E_j^\prime -E_j = L- E_1 - E_2 - E_3$. 
Since the bubble is obtained by rescaling a tubular neighborhood of 
some $2$-sphere representing  $E_1$, and  because 
the Killing fields generating its toric structure are limits of the Killing fields
on $\CP_2 \# 3 \overline{\CP}_2$,  the rescaled region must in fact contain 
 the exceptional divisor $E_1$, as this   curve is one connected component of the 
 zero locus of an appropriate   Killing field. Since the Riemannian diameter of the
  region of curvature concentration 
 tends  to zero  \cite{chenweb}, and since radial geodesics in the Burns metric are length minimizing, even at large radii, the region of curvature concentration can only contain a disk of small intrinsic
 diameter in the curves  $E_2^\prime$ and $E_3^\prime$; thus, the region of
 curvature concentration meets of the locus of exceptional isotropy only
 in $E_1$ and in small adjoining disks in $E_2^\prime$ and $E_3^\prime$. 
 In particular,   $E_2$, $E_3$, and $E_1^\prime$ 
  are  contained in the region of smooth convergence
as $t_j \nearrow 1$. 
These   totally geodesic  submanifolds therefore give rise 
to  totally geodesic  submanifolds in the Gromov-Hausdorff limit. These limit 
submanifolds
are moreover holomorphic curves, since the  original complex structures
converge smoothly to the limit complex structure $J_\infty $ in the region in question; 
and  these three
limit curves  all  have self-intersection $-1$, since
 a tubular neighborhood of each 
original curve 
survives diffeomorphically in the limit. Thus,  $(N,J_\infty)$ is a compact complex
surface which can be blown down, at the limit $E_2$ and $E_3$ 
curves,
to  a compact complex surface with $b_2=1$ containing 
a rational curve of self-intersection $+1$. Surface
classification  \cite[Proposition 4.3]{bpv} now tells us 
that this blow-down must be $\CP_2$. Hence  $(N,J_\infty )$ is actually the blow-up
of $\CP_2$ at two distinct points. Moreover, $(N,J_\infty , g_\infty )$ contains a chain of  three  
 $(-1)$-curves  whose homology classes generate $H_2(N, \RR)$, and the areas of these
curves  are  $\Omega (E_2)$, $\Omega (E_1^\prime)$,  
and  $\Omega (E_3)$.
The  limit extremal K\"ahler metric $g_\infty$ therefore has K\"ahler class 
$\Omega$. Specializing to the case where $\Omega$ is 
the anti-canonical class then proves the claim. 
\end{proof} 

Notice that the above argument actually proves more than what was
initially claimed. We therefore also have the following  result: 

\begin{prop} \label{dolce} 
Let $\Omega $ be any K\"ahler class on  $\CP_2 \# 2 \overline{\CP}_2$ 
for which $${\mathcal T} (\Omega ) < 8.75 = c_1^2 + 1.75~.$$
Then there is an extremal K\"ahler metric $g$ in $\Omega $, and 
 a one-parameter family $g_t$, $t\in [0,1)$ of extremal 
K\"ahler metrics on  $\CP_2 \# 3 \overline{\CP}_2$, with  
 $g_0$ is 
 K\"ahler-Einstein, and with  $g_{t_j} \to g$ in the Gromov-Hausdorff sense for some
$t_j\nearrow 1$. 
 \end{prop}

\pagebreak 
 
\section{Einstein Metrics}

\label{bach}

Let  $(M,J)$  now be a Del Pezzo surface, and  again let ${\zap K} \subset H^2 (M, \RR)$
be its K\"ahler cone. 
  The following variational 
  principle \cite{chenlebweb, lebhem}   unlocks   
  the mysteries of conformally
 K\"ahler, Einstein metrics: 
    
 \begin{prop} \label{critical} 
Suppose that  $h$ is an Einstein metric on $M$ which is conformally related 
to a $J$-compatible K\"ahler metric $g$ with K\"ahler class $[\omega ]\in {\zap K}$. 
Then $[\omega ]$ is a critical point of ${\mathcal A}$. Moreover, $g$ is an 
{extremal}  K\"ahler metric, and  the scalar
curvature $s$ of $g$ is everywhere positive.

Conversely, if $\Omega\in {\zap K}$ is a critical point of $\mathcal A$, 
and if $\omega \in \Omega$ is the K\"ahler form of an extremal K\"ahler metric 
$g$ with scalar curvature $s > 0$, then  $h=s^{-2}g$ is an Einstein metric on $M$. 
 \end{prop}

Because of the crucial role this result plays in our discussion, 
we will now briefly sketch the proof, while referring the reader to
  \cite{chenlebweb} for more details.   
    
For any smooth compact oriented $4$-manifold $M$, one may  consider the 
conformally invariant Riemannian functional 
$$
{\mathcal W}(g)= \int_{M} |W|^{2}_gd\mu_{g}
=  -12\pi^2 \tau (M) + 2\int_M | W_+|^2 d\mu 
$$
where $W$ is the Weyl curvature, and $W_+$ is its self-dual part. 
Following standard convention,  we will call $\mathcal W$ the {\em Weyl functional}. 
The gradient of ${\mathcal W}$ on the space of RIemannian metrics \cite{bach,bes}
 is minus  the {\em Bach tensor} $B$, as defined  by
\begin{eqnarray*}
B_{ab}&=&(\nabla^c\nabla^d+\frac{1}{2}r^{cd})
W_{acbd}\\
&=& (2\nabla^c\nabla^d+r^{cd})(W_+)_{acbd}~.
\end{eqnarray*}
Because $\int|W|^2d\mu$ is invariant under diffeomorphisms and conformal rescalings, 
we have  
$${B_a}^a=0, ~~\nabla^aB_{ab}=0$$
for any $4$-dimensional Riemannian metric.  
The Bianchi identities imply that any Einstein metric 
satisfies the {\em Bach-flat} condition $B=0$. Since  the  Bach-flat condition is 
conformally invariant,   any conformally Einstein $4$-dimensional metric 
 is therefore   \cite{bes,pr2} Bach-flat, too.

We now specialize to the case of K\"ahler metrics. 
For any K\"ahler metric $g$ on a complex surface $(M,J)$, one 
has 
\begin{equation}
\label{sebastian}
|W_+|^2 = \frac{s^2}{24}
\end{equation}
with respect to the  orientation induced by $J$. (Of course, equation (\ref{sebastian}) 
is not conformally invariant --- but  neither is the K\"ahler condition!) 
In particular, any Bach-flat K\"ahler metric is a critical point of the Calabi functional 
$${\mathcal C} (g ) =  \int s^2 d\mu ~,$$
either as a functional on a fixed K\"ahler class $\Omega =[\omega ]$, or
on the entire space of K\"ahler metrics, with $\Omega$ allowed to vary. 
In particular, any Bach-flat K\"ahler metric $g$ must
be an extremal K\"ahler metric, and its K\"ahler class must be 
a critical point of ${\mathcal A}: {\zap K}\to \RR$.

Equation (\ref{sebastian}) reflects the deeper fact that  the self-dual Weyl tensor
of a K\"ahler surface is completely determined by the 
scalar curvature and the K\"ahler form. If 
$(M^4,g,J)$ is a K\"ahler manifold with K\"ahler form $\omega$,  then 
$${(W_+)_{ab}}^{cd}= \frac{s}{12}\left[
\omega_{ab}\omega^{cd} - \delta_{a}^{[c} \delta_{b}^{d]}+ {J_a}^{[c}{J_b}^{d]}
\right]
$$
and plugging this into the formula for the  Bach tensor yields 
\begin{equation}
\label{johann}
B_{ab}= \frac{s}{6}\mathring{r}_{ab} + \frac{1}{4}{J_a}^c{J_b}^d\nabla_c\nabla_d s
+\frac{1}{12} \nabla_a\nabla_b s + \frac{1}{12}g_{ab}\Delta s.
\end{equation}
However, the extremal K\"ahler metric condition is equivalent to 
requiring that the Hessian $\nabla\nabla s$ of the scalar curvature
by $J$-invariant. At an extremal K\"ahler metric, we can therefore 
define  an anti-self-dual $2$-form 
 $\psi$ by  
 $$
 \psi = B (J \cdot , \cdot) =  \frac{1}{6}\Big[ s\rho + 2  i\partial \bar{\partial}s
\Big]_0 
$$
where $\rho$ is the Ricci form, and where 
the subscript  ``$0$'' denotes projection into the primitive $(1,1)$-forms
 $\Lambda^{1,1}_0=\Lambda^-$. 
Since $B$ is symmetric and divergence-free, we thus have 
$$(\delta \psi)_b=-\nabla^a\psi_{ab} = \nabla^a \left( B_{ac}{J_b}^c\right) =  {J_b}^c\nabla^a B_{ac}  =0~,$$
so that  the anti-self-dual $2$-form $\psi$ is
co-closed, and 
 hence  {\em harmonic}.  In particular, the $1$-parameter family $g_t= g + tB$
then  consists entirely of K\"ahler metrics. The first variation ${\mathcal W}$
is $-\int |B|^2 d\mu$ for this family of K\"ahler metrics, and so can vanish only if 
$B=0$. Thus, if $\Omega$ is a K\"ahler class which is a critical point of 
$\mathcal A$, and it $g$ is an extremal K\"ahler metric which belongs to 
$\Omega$, then $g$ is necessarily   Bach-flat.

Now 
 the trace-free Ricci tensor $\mathring{r}$ transforms under conformal changes 
by 
	$$\hat{\mathring{r}}= \mathring{r} +2 u \Hess_0 (u^{-1})~,$$
 and equation (\ref{johann}) tells us that any Bach-flat 
K\"ahler metric satisfies 
$$ \mathring{r}  =  - 2 s^{-1} \Hess_0 (s).$$
Thus, any Bach-flat  K\"ahler metric with $s> 0$ has a conformal rescaling 
$\hat{g} = s^{-2} g$ which is Einstein. 
 Proposition \ref{critical} now follows.

Now let $M$  specifcally be the Del Pezzo surfaces  $\CP_2 \# 2 \overline{\CP}_2$, 
let ${\zap K}\subset H^2(M, \RR)$
be its  K\"ahler cone, and let $\check{\zap K}={\zap K}/\RR^+$,
where the positive reals act by scalar multiplication.  Since the function
${\mathcal T}(\Omega) = (c_1\cdot \Omega)^2/\Omega^2$ is homogeneous
of degree $0$, we now consider it  as a function on $\check{\zap K}$.
 For any $t\in \RR$, let ${\mathbf Y}_t\subset\check{\zap K}$
be the region defined by ${\mathcal T}(\Omega )\leq t$.

\begin{lem} \label{discus}
 If  $7 < t < 8$, then  ${\mathbf Y}_t$ is homeomorphic
to the closed unit $2$-disk,
and so, in particular,  
is  compact. 
\end{lem}
\begin{proof}
On the manifold  $M=\CP_2 \#  \overline{\CP}_2$ under discussion, 
an element of $H^2(M, \RR)$ 
 is  determined 
by the numbers 
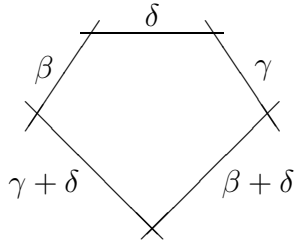
\begin{center}
\begin{picture}(100,80)(0,3)
\put(-7,70){\line(1,0){54}}
\put(40,75){\line(2,-3){28}}
\put(0,75){\line(-2,-3){28}}
\put(68,44){\line(-1,-1){52}}
\put(-28,44){\line(1,-1){52}}
\put(20,77){\makebox(0,0){$\delta$}}
\put(-21,56){\makebox(0,0){$\beta$}}
\put(62,56){\makebox(0,0){$\gamma$}}
\put(-21,13){\makebox(0,0){$\gamma+\delta$}}
\put(60,13){\makebox(0,0){$\beta+\delta$}}
\end{picture}
\end{center}
it assigns to the 
three exceptional curves portrayed as representing  the upper edges
of the pentagon. Since these three numbers represent areas, they all 
 must be positive. Conversely, any cohomology class 
for which these three numbers are positive is actually a K\"ahler class.  Indeed, we can either blow down $M$ along the top curve to obtain $\CP_1\times \CP_1$, or along
the two top diagonal curves to obtain $\CP_2$. 
Pulling back products metrics on $\CP_1\times \CP_1$  then allows one to specify any  desired value of $\beta$, and $\gamma$.  Pulling back a multiple of 
the Fubini-Study metric and adding this on
as well, we can then choose $\delta$ arbitrarily. Hence  a cohomology 
class $\leo \in H^2 (M, \RR)$  is a K\"ahler class iff
it assigns a positive value to each of these three exceptional curves. 
Of course, any cohomology class $\leo$ with this property must consequently 
also satisfy $\leo^2 > 0$ and $c_1\cdot \leo > 0$.

Now, suppose that the convex cone 
in $H^2(M,\RR)$ defined by  
\begin{eqnarray}
\leo^2 &>& 0 \nonumber\\
c_1(M)\cdot \leo &>& 0 \label{lion}\\
(c_1\cdot \leo )^2 &<& 8 \leo^2 \nonumber
\end{eqnarray}
contained a cohomology class that was not a K\"ahler class. By continuity, it would therefore
contain
 a   class $\leo$  which was  non-negative on all three exceptional curves, but 
 which vanished on at least one of them. 
But $\leo$ would then be the pull-back of  some cohomology class $\mho$ on a blow-down 
$N$, obtained by collapsing exactly one exceptional curve,  which satisfied
$\mho^2 > 0$ and 
$$(c_1(N)\cdot \mho )^2 < 8 \mho^2= c_1^2 (N) \mho^2~.$$
But this is a contradiction, because  $c_1(N)$ and  $\mho \in H^2 (N,\RR)$
would then be a pair of time-like vectors in a $2$-dimensional Minkowski space which violated the 
reverse Cauchy-Schwartz inequality for the Lorentzian inner product. Hence
the open convex cone  defined by (\ref{lion}) 
is actually a subset of the  K\"ahler cone ${\zap K}$. 
Consequently, for any $t\in (7 ,  8)$, the set of $\leo \in H^2 (M, \RR)$  with 
\begin{eqnarray}
\leo^2 &>& 0 \nonumber\\
c_1(M)\cdot \leo &>& 0 \label{cub}\\
(c_1\cdot \leo )^2 &\leq& t \leo^2 \nonumber
\end{eqnarray}
consists entirely of K\"ahler classes, and  its quotient  
by  $\RR^+$  therefore   exactly equals ${\mathbf Y}_t$. 
However, this quotient can be identified with
the intersection of  (\ref{cub}) with the hyperplane 
$c_1 \cdot \leo  =7$. 
Writing elements of
this hyperplane uniquely as
 $\leo = c_1 + \eta$,
where $c_1\cdot \eta = 0$, we thus have identified
${\mathbf Y}_t$ with the closed ball
$$|\eta^2| \leq  \frac{7(t-7)}{t}$$
in the $2$-dimensional space-like hyperplane $c_1^\perp\subset H^2(M, \RR)$.
This proves the claim.   \end{proof}

Theorem \ref{pace} is now an easy consequence. Indeed, because
Lemma \ref{up1} tells us that 
$0 \leq \mathcal B < \frac{1}{4}$ on the entire K\"ahler cone, 
the infimum of $\mathcal A$ for $M=\CP_2 \# 2 \overline{\CP}_2$ must be 
less than 
${\mathcal T}(c_1) + \frac{1}{4} = c_1^2+ \frac{1}{4}= 7 \frac{1}{4}$,
whereas ${\mathcal A} \geq 7 \frac{1}{4}$ outside the interior
of ${\mathbf Y}_{c_1^2+\frac{1}{4}}$. Since ${\mathbf Y}_{c_1^2+\frac{1}{4}}$ is compact by 
Lemma \ref{discus},  there is consequently  an
interior point $\check{\Omega}$ of  ${\mathbf Y}_{c_1^2+\frac{1}{4}}$ at which $\mathcal A$
achieves its minimum. Notice that  $\check{\Omega}$ is a critical point of $\mathcal A$, and let 
$\Omega$ be a K\"ahler class which projects to 
$\check{\Omega}$. By Theorem~\ref{laudate}, 
$\Omega$ is represented by an extremal K\"ahler metric $g$, and by Lemma \ref{pos2},
this extremal K\"ahler metric has positive scalar curvature $s > 0$. 
Proposition \ref{critical} then tells us that $h= s^{-2}g$ is
an Einstein metric on $M$, and, by construction, $h$ minimizes the Weyl
functional among all conformally K\"ahler metrics on $M$. 
We have thus succeeded in proving Theorem \ref{pace}. 
Theorem  \ref{quoniam}   now similarly  follows  from Proposition \ref{dolce},
since $7.25 < 8.75$.

\pagebreak 

\appendix
\makeatletter
\def\@seccntformat#1{Appendix\ \csname the#1\endcsname :\quad}
\makeatother

\section{ Computations for $\CP\# 2 \overline{\CP}_2$} \label{comp2} 

In this section, we will  use a combination of elementary  symplectic geometry 
and  computer-assisted  algebra 
to  estimate some key geometric invariants of 
extremal K\"ahler metrics on $M=\CP\# 2 \overline{\CP}_2$ that are needed 
in the body of the paper. We begin
by fixing a K\"ahler class, normalized by rescaling so that the proper
transform of the projective line between the two blow-up points has area $1$:

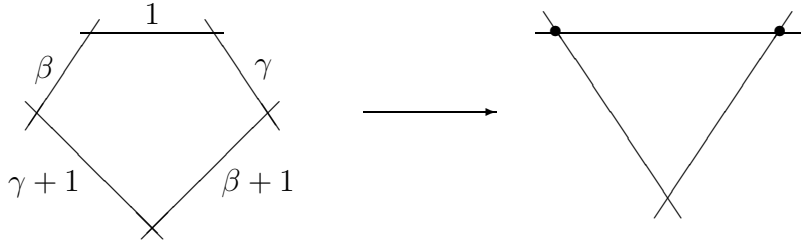
\begin{center}
\begin{picture}(240,80)(0,3)
\put(-7,70){\line(1,0){54}}
\put(40,75){\line(2,-3){28}}
\put(0,75){\line(-2,-3){28}}
\put(68,44){\line(-1,-1){52}}
\put(-28,44){\line(1,-1){52}}
\put(20,77){\makebox(0,0){$1$}}
\put(-21,56){\makebox(0,0){$\beta$}}
\put(62,56){\makebox(0,0){$\gamma$}}
\put(-21,13){\makebox(0,0){$\gamma+1$}}
\put(60,13){\makebox(0,0){$\beta+1$}}
\put(100,40){\vector(1,0){50}}
\put(210,0){\line(2,3){52}}
\put(220,0){\line(-2,3){52}}
\put(165,70){\line(1,0){100}}
\put(172.6,70.5){\circle*{4}}
\put(257.4,70.5){\circle*{4}}
\end{picture}
\end{center}
Take the two blow-up points to be $[1,0,0], [0,1,0] \in \CP_2$, 
and fix the  maximal torus 
$$\left[ \begin{array}{ccc}
e^{i\theta}&&\\&e^{i\phi}&\\&&1
\end{array}\right]$$
in the automophism group. Then, for any $T^2$-invariant metric, 
the moment map of the torus action will take values in a pentagon,
which after translation becomes the following: 
 
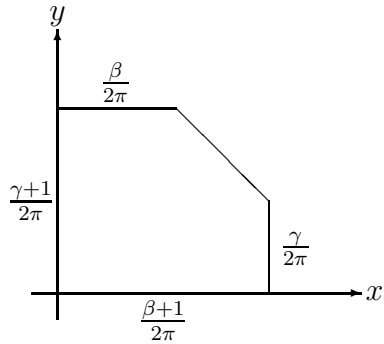
\begin{center}
\begin{picture}(120,120)(0,0)
\put(0,10){\vector(1,0){125}}
\put(10,0){\vector(0,1){110}}
\put(130,10){\makebox(0,0){$x$}}
\put(10,115){\makebox(0,0){$y$}}
\put(50,0){\makebox(0,0){$\frac{\beta +1}{2\pi}$}}
\put(0,45){\makebox(0,0){$\frac{\gamma +1}{2\pi}$}}
\put(100,27){\makebox(0,0){$\frac{\gamma }{2\pi}$}}
\put(32,90){\makebox(0,0){$\frac{\beta }{2\pi}$}}
\put(10,80){\line(1,0){45}}
\put(55,80){\line(1,-1){35}}
\put(90,10){\line(0,1){35}}
\end{picture}
\end{center}
Let ${\mathfrak F}_1$ and ${\mathfrak F}_2$ be Futaki invariants of
this K\"ahler class with respect to the vector fields with Hamiltonians 
$-x$ and $-y$. Then \cite{ls} for any $T^2$-invariant metric, 
\begin{eqnarray*}
{\mathfrak F}_1&=& \int_M x(s-s_0) d\mu \\
&=& \frac{1}{V}\left[ 
(\beta -2\gamma) (\frac{1}{3} + \gamma + \gamma^2) + \gamma (\gamma - \beta) (2+ \beta + 2\gamma ) 
\right]\\
{\mathfrak F}_2&=& \int_M y(s-s_0) d\mu \\
&=& \frac{1}{V}\left[ 
(\gamma -2\beta) (\frac{1}{3} + \beta + \beta^2) + \beta (\beta - \gamma) (2+ \gamma + 2\beta ) 
\right]\\
\end{eqnarray*}
where 
$$
V= \beta \gamma + \beta + \gamma +  \frac{1}{2}. 
$$

Note that, by Archimedes' principle, 
 the push-forward of the volume measure of $M$ is exactly $4\pi^2$
times the Euclidean measure  on the moment polygon. 
Thus, for example, the average values  $x_0$ and $y_0$ of the Hamiltonians  
 $x$ and $y$ on $M$ are also  the $x$ and $y$ coordinates of
 the barycenter of the moment pentagon. This same observation also 
 makes it straightforward to compute
 the following useful constants:
\begin{eqnarray*}
A&:=& \int_M (x-x_0)^2d\mu \\
&=& \frac{1 +  6 (1 + \beta)[  \beta + \beta^2 + \beta^3+
    \gamma  (1 + 4 \beta + 4 \beta^2 + 2 \beta^3)+
    \gamma^2 (1 + \beta)^3]}{288\pi^2 V} \\
   B&:=&
    \int_M (y-y_0)^2d\mu\\
   &=&  \frac{1 +  6 (1 + \gamma)[  \gamma + \gamma^2 + \gamma^3+
    \beta  (1 + 4 \gamma + 4 \gamma^2 + 2 \gamma^3)+
    \beta^2 (1 + \gamma)^3] }{288\pi^2 V}
  \\  C &:=&  \int_M (x-x_0)(y-y_0)d\mu\\
 &=& -\frac{1 + 6 (1 + \beta) (1+\gamma) (\beta+ \gamma  + 3 \beta   \gamma)}{576\pi^2 V}
\end{eqnarray*}
If our metric is  extremal, we then have 
\begin{equation}
\label{scal2}
s-s_0= a(x-x_0) + b (y-y_0)
\end{equation}
where the constants $a$ and $b$ are  given by 
\begin{eqnarray*}
a &=& \frac{B{\mathfrak F}_1- C{\mathfrak F}_2}{AB-C^2}\\
b&=&\frac{ A{\mathfrak F}_2-C{\mathfrak F}_1}{AB-C^2}~.
\end{eqnarray*}
Consequently, 
$$
\int_M (s-s_0)^2 d\mu = \frac{B{\mathfrak F}_1^2 - 2C{\mathfrak F}_1{\mathfrak F}_2+A{\mathfrak F}_2^2 }{AB-C^2}
$$
for any extremal K\"ahler metric, 
 and even without assuming the existence of an extremal K\"ahler metric our arguments
therefore  assign a prominent role to the quantity
\begin{eqnarray*}
{\mathcal B}(\Omega)  &=& \frac{1}{32\pi^2}\frac{B{\mathfrak F}_1^2 - 2C{\mathfrak F}_1{\mathfrak F}_+A{\mathfrak F}_2^2 }{AB-C^2}\\
&=& 
8 \Big[ \gamma^2 (1 + 4 \gamma + 6 \gamma^2 + 4 \gamma^3) + \\&&
     \beta \gamma (-1 + 3 \gamma + 18 \gamma^2 + 26 \gamma^3 + 16 \gamma^4) + \\&&
     2 \beta^5 (2 + 8 \gamma + 21 \gamma^2 + 33 \gamma^3 + 27 \gamma^4 + 9 \gamma^5) + \\&&
     \beta^2 (1 + 3 \gamma + 27 \gamma^2 + 79 \gamma^3 + 89 \gamma^4 + 42 \gamma^5) + \\&&
     \beta^4 (6 + 26 \gamma + 89 \gamma^2 + 168 \gamma^3 + 150 \gamma^4 + 54 \gamma^5) + \\&&
     \beta^3 (4 + 18 \gamma + 79 \gamma^2 + 173 \gamma^3 + 168 \gamma^4 +     66 \gamma^5)\Big] \Big/ \\&&
        \Big[ 48 \beta^6 (1 + \gamma)^6 + 
   48 \beta^5 (1 + \gamma)^3 (3 + 12 \gamma + 14 \gamma^2 + 6 \gamma^3) + \\&&(1 + 2 \gamma)^2 (1 + 
      8 \gamma + 20 \gamma^2 + 24 \gamma^3 + 12 \gamma^4) + \\&&
   4 \beta^4 (1 + \gamma)^2 (47 + 282 \gamma + 573 \gamma^2 + 504 \gamma^3 + 180 \gamma^4) + \\&&
   4 \beta (3 + 33 \gamma + 140 \gamma^2 + 306 \gamma^3 + 376 \gamma^4 + 252 \gamma^5 + 72 \gamma^6) + \\&&
   8 \beta^2 (7 + 70 \gamma + 270 \gamma^2 + 535 \gamma^3 + 592 \gamma^4 + 354 \gamma^5 + 
      90 \gamma^6) + \\&&
   8 \beta^3 (17 + 153 \gamma + 535 \gamma^2 + 963 \gamma^3 + 966 \gamma^4 + 522 \gamma^5 + 
      120 \gamma^6)\Big] 
\end{eqnarray*}

\begin{lem} \label{up1} 
One has ${\mathcal  B} < \frac{1}{4}$ throughout the Kahler cone of $\CP\# 2 \overline{\CP}_2$. \end{lem}
\begin{proof}
Subtracting $4$ times the numerator of the above expression from the denominator 
yields 
\begin{eqnarray*}
&&1 + 12 \gamma + 24 \gamma^2 + 8 \gamma^3 - 4 \gamma^4 + 16 \gamma^5 + 48 \gamma^6 + 
 48 \beta^6 (1 + \gamma)^6 + \\&&
 16 \beta^5 (1 + 31 \gamma + 93 \gamma^2 + 129 \gamma^3 + 108 \gamma^4 + 60 \gamma^5 + 18 \gamma^6) + \\&&
  4 \beta (3 + 41 \gamma + 116 \gamma^2 + 162 \gamma^3 + 168 \gamma^4 + 124 \gamma^5 + 72 \gamma^6) + \\&&
 8 \beta^2 (3 + 58 \gamma + 162 \gamma^2 + 219 \gamma^3 + 236 \gamma^4 + 186 \gamma^5 + 90 \gamma^6) + \\&&
 8 \beta^3 (1 + 81 \gamma + 219 \gamma^2 + 271 \gamma^3 + 294 \gamma^4 + 258 \gamma^5 + 120 \gamma^6) + \\&&
 4 \beta^4 (-1 + 168 \gamma + 472 \gamma^2 + 588 \gamma^3 + 561 \gamma^4 + 432 \gamma^5 + 180 \gamma^6) .
 \end{eqnarray*}
 Term by term,  this is  greater than  
 $4(\gamma^2-\gamma^4+\gamma^6+\beta^2-\beta^4+\beta^6)>0$. Thus 
the denominator of our expression for 
$\mathcal B$ is more than four times larger than the corresponding numerator.
Hence ${\mathcal B} < \frac{1}{4}$, as claimed. 
\end{proof}

The coefficient $a$  of equation (\ref{scal2}) is explicitly given by

\bigskip

 \noindent 
 $\displaystyle 
 -192  \pi^2 \gamma\Big[1 + 4 \gamma + 6 \gamma^2 + 4 \gamma^3 + 6 \beta^3 (1 + \gamma)^3 + 
2 \beta^2 (6 + 18 \gamma + 17 \gamma^2 + 6 \gamma^3) + 
     \beta (7 + 21 \gamma + 22 \gamma^2 + 10 \gamma^3)\Big] \Big/
   \Big[1 + 10 \gamma + 36 \gamma^2 + 
   64 \gamma^3 + 60 \gamma^4 + 24 \gamma^5 + 
  24 \beta^5 (1 + \gamma)^5 + 
  12 \beta^4 (1 + \gamma)^2 (5 + 20 \gamma + 23 \gamma^2 + 10 \gamma^3) + 
  16 \beta^3 (4 + 28 \gamma + 72 \gamma^2 + 90 \gamma^3 + 57 \gamma^4 + 15 \gamma^5) + 
 12 \beta^2 (3 + 24 \gamma + 69 \gamma^2 + 96 \gamma^3 + 68 \gamma^4 + 20 \gamma^5) + 
 2 \beta (5 + 45 \gamma + 144 \gamma^2 + 224 \gamma^3 + 180 \gamma^4 + 60 \gamma^5)\Big] 
 $

\bigskip

 \noindent
   and $b$ is given by the analogous expression with $\beta$ and $\gamma$
   interchanged. In particular, both of these coefficients are always negative. 
 
 \begin{lem}\label{pos2}
If $g$ is an extremal K\"ahler metric on $M= \CP_2 \# 2 \overline{\CP}_2$, 
then the scalar curvature
$s$ of $g$ is positive at every point of  $M$.  Moreover, there is a 
smooth function  $f : {\zap K}\to \RR$   such that 
$s_{\max} = f (\Omega )$ for any extremal K\"ahler metric. 
\end{lem}
\begin{proof}
Since $a$ and $b$ are negative, the values of $s_0+ a (x-x_0) + b(y-y_0)$
at $(0,0)$ and $(\frac{\beta}{2\pi}, \frac{\gamma}{2\pi})$ are certainly
upper and lower bounds for $s$.  Making the substitution
$$
s_0 = 4\pi \frac{c_1\cdot \Omega}{V}= 4\pi  \frac{3 + 2 \beta + 2 \gamma}{\frac{1}{2} + \beta + \gamma +
 \beta\gamma}
$$
into the value at $(\frac{\beta}{2\pi}, \frac{\gamma}{2\pi})$ 
gives us the positive lower bound 

\begin{eqnarray*}
s_{\min} \geq
&24\pi& \Big[ (1 + 2 \gamma) (1 + 2 \gamma + 2 \gamma^2)^2 + 8 \beta^5 (1 + \gamma)^4 +\\&&
     4 \beta^4 (5 + 24 \gamma + 40 \gamma^2 + 32 \gamma^3 + 13 \gamma^4 + 2 \gamma^5) + \\&&
     8 \beta^3 (3 + 14 \gamma + 25 \gamma^2 + 26 \gamma^3 + 16 \gamma^4 + 4 \gamma^5) + \\&&
     4 \beta^2 (4 + 16 \gamma + 33 \gamma^2 + 50 \gamma^3 + 40 \gamma^4 + 12 \gamma^5) + 
     \\&&
     2 \beta (3 + 12 \gamma + 32 \gamma^2 + 56 \gamma^3 + 48 \gamma^4 + 16 \gamma^5)\Big]\Big/
     \\&&
     \Big[ 1 + 
   10 \gamma + 36 \gamma^2 + 64 \gamma^3 + 60 \gamma^4 + 24 \gamma^5 +
   24 \beta^5 (1 + \gamma)^5 +  \\&&
   12 \beta^4 (1 + \gamma)^2 (5 + 20 \gamma + 23 \gamma^2 + 10 \gamma^3) + \\&&
   16 \beta^3 (4 + 28 \gamma + 72 \gamma^2 + 90 \gamma^3 + 57 \gamma^4 + 15 \gamma^5) +
    \\&&
   12 \beta^2 (3 + 24 \gamma + 69 \gamma^2 + 96 \gamma^3 + 68 \gamma^4 + 20 \gamma^5) +
    \\&&
   2 \beta (5 + 45 \gamma + 144 \gamma^2 + 224 \gamma^3 + 180 \gamma^4 + 60 \gamma^5)\Big]
\end{eqnarray*}
while making the same substitution into the value at 
 $(0,0)$  gives us a smooth function $f$ with 
 $f(\Omega)= s_{\max }$ for any extremal K\"ahler metric, and the  
 requirement that $f$ be homogeneous of degree $-1$ then specifies an appropriate  smooth
extension of $f$  to the entire K\"ahler cone. 
\end{proof}

 \pagebreak 

\section{ Computations for $\CP\# 3 \overline{\CP}_2$} \label{comp3}

We now carry out computations analogous to those in the previous appendix, but this time for 
$M=\CP\# 3 \overline{\CP}_2$. First recall  that the general K\"ahler class
on this manifold is determined by four real numbers:

 
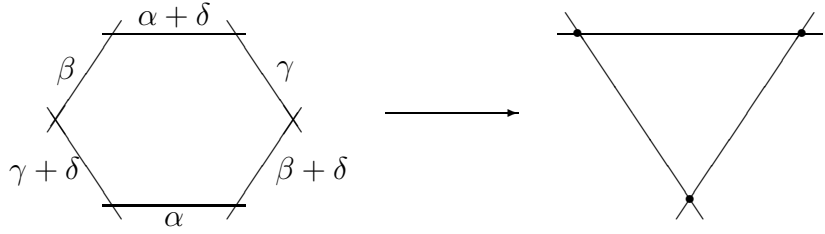
\begin{center}
\begin{picture}(240,80)(0,3)
\put(-7,70){\line(1,0){54}}
\put(40,75){\line(2,-3){28}}
\put(0,75){\line(-2,-3){28}}
\put(-7,5){\line(1,0){54}}
\put(40,0){\line(2,3){28}}
\put(0,0){\line(-2,3){28}}
\put(20,0){\makebox(0,0){$\alpha$}}
\put(20,77){\makebox(0,0){$\alpha+\delta$}}
\put(-21,56){\makebox(0,0){$\beta$}}
\put(62,56){\makebox(0,0){$\gamma$}}
\put(-29,18){\makebox(0,0){$\gamma+\delta$}}
\put(72,18){\makebox(0,0){$\beta+\delta$}}
\put(100,40){\vector(1,0){50}}
\put(210,0){\line(2,3){52}}
\put(220,0){\line(-2,3){52}}
\put(165,70){\line(1,0){100}}
\put(215,7.5){\circle*{3}}
\put(172.6,70.5){\circle*{3}}
\put(257.4,70.5){\circle*{3}}
\end{picture}
\end{center}
By applying a Cremona transformation, we may also assume that $\delta \geq 0$.
After rescaling, the region $\delta > 0$ can then be completely understood
in terms of those classes
% \begin{center}
%\begin{picture}(100,80)(0,3)
%\put(-7,70){\line(1,0){54}}
%\put(40,75){\line(2,-3){28}}
%\put(0,75){\line(-2,-3){28}}
%\put(-7,5){\line(1,0){54}}
%\put(40,0){\line(2,3){28}}
%\put(0,0){\line(-2,3){28}}
%\put(20,0){\makebox(0,0){$\alpha$}}
%\put(20,77){\makebox(0,0){$\alpha+1$}}
%\put(-21,56){\makebox(0,0){$\beta$}}
%\put(62,56){\makebox(0,0){$\gamma$}}
%\put(-29,18){\makebox(0,0){$\gamma+1$}}
%\put(72,18){\makebox(0,0){$\beta+1$}}
%\end{picture}
%\end{center}
for which $\delta =1$; these are exactly parameterized by the 
three arbitrary positive real numbers $\alpha$, $\beta$, and $\gamma$. 
Of course, any invariant geometrical conclusion we reach regarding 
 this region will automatically
also apply to ``mirror'' region reached by the Cremona transformation.
This will allow us to understand the entire K\"ahler cone ${\zap K}$, 
as long as we are careful to also  account for  
 the hyperplane $\delta=0$. 
 
 We now once again fix the $2$-torus in the automorphims group 
 corresponding to $[z_1:z_2:z_3]\mapsto [e^{i\theta}z_1:e^{i\phi}z_2:z_3]$.
 The image of $M$ under the moment map is then the hexagon 
 
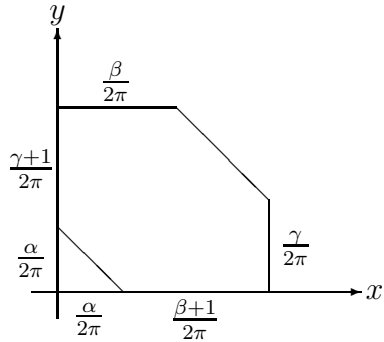
\begin{center}
\begin{picture}(120,120)(0,0)
\put(0,10){\vector(1,0){125}}
\put(10,0){\vector(0,1){110}}
\put(130,10){\makebox(0,0){$x$}}
\put(10,115){\makebox(0,0){$y$}}
\put(22,0){\makebox(0,0){$\frac{\alpha}{2\pi}$}}
\put(0,22){\makebox(0,0){$\frac{\alpha}{2\pi}$}}
\put(62,0){\makebox(0,0){$\frac{\beta +1}{2\pi}$}}
\put(0,57){\makebox(0,0){$\frac{\gamma +1}{2\pi}$}}
\put(100,27){\makebox(0,0){$\frac{\gamma }{2\pi}$}}
\put(32,90){\makebox(0,0){$\frac{\beta }{2\pi}$}}
\put(10,35){\line(1,-1){25}}
\put(10,80){\line(1,0){45}}
\put(55,80){\line(1,-1){35}}
\put(90,10){\line(0,1){35}}
\end{picture}
\end{center}
and our formulas \cite{ls} for the components of the Futaki invariant become 
\begin{eqnarray*}
{\mathfrak F}_1&=& \int_M x(s-s_0) d\mu \\
&=& \frac{1}{V}\left[ 
(\alpha+\beta -2\gamma) (\frac{1}{3} + \gamma + \gamma^2) + (\gamma-\alpha) (\gamma - \beta) (2+\alpha+ \beta + 2\gamma ) 
\right]\\
{\mathfrak F}_2&=& \int_M y(s-s_0) d\mu \\
&=& \frac{1}{V}\left[ 
(\alpha+ \gamma -2\beta) (\frac{1}{3} + \beta + \beta^2) + (\beta-\alpha) (\beta - \gamma) (2+ \alpha+\gamma + 2\beta ) 
\right]\\
\end{eqnarray*}
where 
$$
V= \alpha \beta+ \alpha \gamma + \beta \gamma + \alpha + \beta + \gamma +  \frac{1}{2}
$$
is the volume of $(M,\Omega)$. 
Three other essential coefficients needed in our computation are 
\begin{eqnarray*}
A&:=& \int_M (x-x_0)^2d\mu
% \\
%&=&
%\Big[1 + 6 \gamma  + 6 \gamma ^2 + 6 \beta ^4 (1 + \gamma )^2 + 6 \alpha^4 (1 + \beta  + \gamma )^2 + 
%  12 \beta ^3 (1 + 3 \gamma  + 2 \gamma ^2) + 
%  \\&& 12 \beta ^2 (1 + 4 \gamma  + 3 \gamma ^2) + 
%  6 \beta  (1 + 5 \gamma  + 4 \gamma ^2) + 
%  12 \alpha^3 (1 + 3 \gamma  + 2 \gamma ^2 + 2 \beta ^2 (1 + \gamma ) +
%  \\&&  \beta  (3 + 6 \gamma  + 2 \gamma ^2)) + 
%  6 \alpha^2 (2 + \beta ^4 + 8 \gamma  + 6 \gamma ^2 + 4 \beta ^3 (1 + \gamma ) + 
%     3 \beta ^2 (3 + 6 \gamma  + 2 \gamma ^2) + 
%     \\&& 4 \beta  (2 + 6 \gamma  + 3 \gamma ^2)) + 
%  6 \alpha [1 + 5 \gamma  + 4 \gamma ^2 + 2 \beta ^4 (1 + \gamma ) + 2 \beta ^3 (3 + 6 \gamma  + 2 \gamma ^2) + 
%  \\&&    4 \beta ^2 (2 + 6 \gamma  + 3 \gamma ^2) + \beta  (5 + 20 \gamma  + 12 \gamma^2)\Big]
%  \Big/ 288\pi^2 V
 \\  &=& (288\pi^2 V)^{-1}
   \Big[1 + 6 \beta + 12 \beta^2 + 12 \beta^3 + 6 \beta^4 + 
  6 \gamma^2 (1 + \beta)^4 + 
  6 \alpha^4 (1 + \gamma + \beta)^2 +
    \\&& 6 \gamma (1 + 5 \beta + 8 \beta^2 + 6 \beta^3 + 2 \beta^4) + 
  6 \alpha^2 (2 + 8 \beta + 9 \beta^2 + 4 \beta^3 + \beta^4 + 6 \gamma^2 (1 + \beta)^2 + 
\\&&     2 \gamma (2 + \beta)^2 (1 + 2 \beta)) + 
  12 \alpha^3 (1 + 3 \beta + 2 \beta^2 + 2 \gamma^2 (1 + \beta) + \gamma (3 + 6 \beta + 2 \beta^2)) + 
\\&&   6 \alpha (1 + 5 \beta + 8 \beta^2 + 6 \beta^3 + 2 \beta^4 + 4 \gamma^2 (1 + \beta)^3 + 
     \gamma (5 + 20 \beta + 24 \beta^2 + 12 \beta^3 + 2 \beta^4))\Big]
  \end{eqnarray*}
 \begin{eqnarray*} 
      B&:=&
    \int_M (y-y_0)^2d\mu\\
   &=& (288\pi^2 V)^{-1}
   \Big[1 + 6 \gamma + 12 \gamma^2 + 12 \gamma^3 + 6 \gamma^4 + 
  6 \beta^2 (1 + \gamma)^4 + 
  6 \alpha^4 (1 + \beta + \gamma)^2 +
    \\&& 6 \beta (1 + 5 \gamma + 8 \gamma^2 + 6 \gamma^3 + 2 \gamma^4) + 
  6 \alpha^2 (2 + 8 \gamma + 9 \gamma^2 + 4 \gamma^3 + \gamma^4 + 6 \beta^2 (1 + \gamma)^2 + 
\\&&     2 \beta (2 + \gamma)^2 (1 + 2 \gamma)) + 
  12 \alpha^3 (1 + 3 \gamma + 2 \gamma^2 + 2 \beta^2 (1 + \gamma) + \beta (3 + 6 \gamma + 2 \gamma^2)) + 
\\&&   6 \alpha (1 + 5 \gamma + 8 \gamma^2 + 6 \gamma^3 + 2 \gamma^4 + 4 \beta^2 (1 + \gamma)^3 + 
     \beta (5 + 20 \gamma + 24 \gamma^2 + 12 \gamma^3 + 2 \gamma^4))\Big]
\end{eqnarray*}   
  and 
 \begin{eqnarray*} 
   C &:=&  \int_M (x-x_0)(y-y_0)d\mu\\
 &=& -(576\pi^2 V)^{-1} \Big[1 + 6 \gamma + 6 \gamma^2 + 12 \alpha^4 (1 + \beta + \gamma)^2 +
  6 \beta^2 (1 + 4 \gamma + 3 \gamma^2) + 
\\&&  6 \beta (1 + 5 \gamma + 4 \gamma^2) + 
 24 \alpha^3 (1 + 3 \gamma + 2 \gamma^2 + 2 \beta^2 (1 + \gamma) + \beta (3 + 6 \gamma + 2 \gamma^2)) + 
 \\&& 18 \alpha^2 (1 + 4 \gamma + 3 \gamma^2 + \beta^2 (3 + 6 \gamma + 2 \gamma^2) + 
    2 \beta (2 + 6 \gamma + 3 \gamma^2)) + 
\\&&  6 \alpha (1 + 5 \gamma + 4 \gamma^2 + 2 \beta^2 (2 + 6 \gamma + 3 \gamma^2) + 
    \beta (5 + 20 \gamma + 12 \gamma^2))\Big]
\end{eqnarray*}
If our metric is  extremal, we once again have 
\begin{equation}
\label{scal3}
s-s_0= a(x-x_0) + b (x-x_0)
\end{equation}
for the constants 
\begin{eqnarray*}
a &=& \frac{B{\mathfrak F}_1- C{\mathfrak F}_2}{AB-C^2}\\
b&=&\frac{-C{\mathfrak F}_1+ A{\mathfrak F}_2}{AB-C^2}~.
\end{eqnarray*}
and hence 
$$
\int_M (s-s_0)^2 d\mu = \frac{B{\mathfrak F}_1^2 - 2C{\mathfrak F}_1{\mathfrak F}_+A{\mathfrak F}_2^2 }{AB-C^2}~.
$$
For  any K\"ahler class $\Omega=[\omega ]$ on  $\CP\# 3 \overline{\CP}_2$ with $\delta=1$, 
$\mathcal B (\Omega)$ is the right-hand side over $32\pi^2$, so  automated 
calculation reveals that $\mathcal B (\Omega)$   equals 

\bigskip 

\noindent 
$
 \Big[ \gamma^2 (1 + 4 \gamma + 6 \gamma^2 + 4 \gamma^3) + 
     \beta \gamma (-1 + 3 \gamma + 18 \gamma^2 + 26 \gamma^3 + 16 \gamma^4) + 
     2 \beta^5 (2 + 8 \gamma + 21 \gamma^2 + 33 \gamma^3 + 27 \gamma^4 + 9 \gamma^5) + 
     \beta^2 (1 + 3 \gamma + 27 \gamma^2 + 79 \gamma^3 + 89 \gamma^4 + 42 \gamma^5) + 
     \beta^4 (6 + 26 \gamma + 89 \gamma^2 + 168 \gamma^3 + 150 \gamma^4 + 54 \gamma^5) + 
        \beta^3 (4 + 18 \gamma + 79 \gamma^2 + 173 \gamma^3 + 168 \gamma^4 + 66 \gamma^5) + 
        2 \alpha^5 (2 + 9 \beta^5 + 8 \gamma + 21 \gamma^2 + 33 \gamma^3 + 27 \gamma^4 + 9 \gamma^5 + 
        9 \beta^4 (3 + \gamma) + 3 \beta^2 (7 + 5 \gamma) + 
        3 \beta^3 (11 + 6 \gamma) + 
        \beta (8 + 12 \gamma + 15 \gamma^2 + 18 \gamma^3 + 9 \gamma^4)) + 
        \alpha^4 \Big( 168 \gamma^3 + 150 \gamma^4 + 54 \gamma^5 + 
        18 \beta^5 (3 + \gamma) + 
         \beta^4 (150 + 72 \gamma - 18 \gamma^2) +  
        6 \beta^3 (28 + 12 \gamma - 15 \gamma^2 - 6 \gamma^3) +
         6 + 26 \gamma + 89 \gamma^2 +
         \beta^2 (89 - 6 \gamma - 162 \gamma^2 - 90 \gamma^3 - 18 \gamma^4) + 
        2 \beta (13 + 2 \gamma - 3 \gamma^2 + 36 \gamma^3 + 36 \gamma^4 + 9 \gamma^5)\Big) + 
        \alpha^2 \Big(1 + 3 \gamma + 27 \gamma^2 + 79 \gamma^3 + 89 \gamma^4 + 42 \gamma^5 + 
        6 \beta^5 (7 + 5 \gamma) - 
         \beta^4 (-89 + 6 \gamma + 162 \gamma^2 + 90 \gamma^3 + 18 \gamma^4) + 
          3 \beta^2 (9 - 56 \gamma - 165 \gamma^2 - 144 \gamma^3 - 54 \gamma^4) - 
        \beta^3 (-79 + 111 \gamma + 432 \gamma^2 + 324 \gamma^3 + 90 \gamma^4) + 
        3 \beta (1 - 23 \gamma - 56 \gamma^2 - 37 \gamma^3 - 2 \gamma^4 + 10 \gamma^5)\Big) + 
        \alpha^3 \Big(4 + 18 \gamma + 79 \gamma^2 + 173 \gamma^3 + 168 \gamma^4 + 66 \gamma^5 + 
        6 \beta^5 (11 + 6 \gamma) +
          6 \beta^4 (28  12 \gamma - 15 \gamma^2 - 6 \gamma^3) +        
        \beta^3 (173 - 324 \gamma^2 - 216 \gamma^3 - 36 \gamma^4) - 
        \beta^2 (-79 + 111 \gamma + 432 \gamma^2 + 324 \gamma^3 + 90 \gamma^4) + 
        \beta (18 - 46 \gamma - 111 \gamma^2 + 72 \gamma^4 + 36 \gamma^5)\Big) + 
        \alpha \Big(2 \beta^5 (8 + 12 \gamma + 15 \gamma^2 + 18 \gamma^3 + 9 \gamma^4) + 
        \gamma (-1 + 3 \gamma + 18 \gamma^2 + 26 \gamma^3 + 16 \gamma^4) + 
          2 \beta^4 (13 + 2 \gamma - 3 \gamma^2 + 36 \gamma^3 + 36 \gamma^4 + 9 \gamma^5) + 
        3 \beta^2 (1 - 23 \gamma - 56 \gamma^2 - 37 \gamma^3 - 2 \gamma^4 + 10 \gamma^5) + 
          \beta (-1 - 30 \gamma - 69 \gamma^2 - 46 \gamma^3 + 4 \gamma^4 + 24 \gamma^5) + 
        \beta^3 (18 - 46 \gamma - 111 \gamma^2 + 72 \gamma^4 + 36 \gamma^5)\Big)\Big]
        \Big/
 \\          \Big[(1 + 2 \gamma + 
     2 \beta (1 + \gamma) + 2 \alpha (1 + \beta + \gamma)) (1 + 10 \gamma + 36 \gamma^2 + 64 \gamma^3 + 
       60 \gamma^4 + 24 \gamma^5 + 
         24 \beta^5 (1 + \gamma)^5 + 24 \alpha^5 (1 + \beta + \gamma)^5 + 
     12 \beta^4 (1 + \gamma)^2 (5 + 20 \gamma + 23 \gamma^2 + 10 \gamma^3) + 
        16 \beta^3 (4 + 28 \gamma + 72 \gamma^2 + 90 \gamma^3 + 57 \gamma^4 + 15 \gamma^5) + 
     12 \beta^2 (3 + 24 \gamma + 69 \gamma^2 + 96 \gamma^3 + 68 \gamma^4 + 20 \gamma^5) + 
        2 \beta (5 + 45 \gamma + 144 \gamma^2 + 224 \gamma^3 + 180 \gamma^4 + 60 \gamma^5) +    12 \alpha^4 (1 + \beta + \gamma)^2 (5 + 20 \gamma + 23 \gamma^2 + 10 \gamma^3 + 
        10 \beta^3 (1 + \gamma) + \beta^2 (23 + 46 \gamma + 16 \gamma^2) + 
         2 \beta (10 + 30 \gamma + 23 \gamma^2 + 5 \gamma^3)) + 
       16 \alpha^3 (4 + 28 \gamma + 72 \gamma^2 + 90 \gamma^3 + 57 \gamma^4 + 15 \gamma^5 + 
        15 \beta^5 (1 + \gamma)^2 + 
        3 \beta^4 (19 + 57 \gamma + 50 \gamma^2 + 13 \gamma^3) + 
          3 \beta^3 (30 + 120 \gamma + 155 \gamma^2 + 78 \gamma^3 + 13 \gamma^4) + 
          3 \beta^2 (24 + 120 \gamma + 206 \gamma^2 + 155 \gamma^3 + 50 \gamma^4 + 5 \gamma^5) + 
        \beta (28 + 168 \gamma + 360 \gamma^2 + 360 \gamma^3 + 171 \gamma^4 + 30 \gamma^5)) + 
        12 \alpha^2 (3 + 24 \gamma + 69 \gamma^2 + 96 \gamma^3 + 68 \gamma^4 + 20 \gamma^5 + 
        20 \beta^5 (1 + \gamma)^3 + 
          \beta^4 (68 + 272 \gamma + 366 \gamma^2 + 200 \gamma^3 + 36 \gamma^4) + 
        4 \beta^3 (24 + 120 \gamma + 206 \gamma^2 + 155 \gamma^3 + 50 \gamma^4 + 5 \gamma^5) + 
          2 \beta (12 + 84 \gamma + 207 \gamma^2 + 240 \gamma^3 + 136 \gamma^4 + 30 \gamma^5) + 
        \beta^2 (69 + 414 \gamma + 864 \gamma^2 + 824 \gamma^3 + 366 \gamma^4 + 60 \gamma^5)) + 
        2 \alpha (  
        60 \beta^5 (1 + \gamma)^4 + 
        12 \beta^4 (15 + 75 \gamma + 136 \gamma^2 + 114 \gamma^3 + 43 \gamma^4 + 5 \gamma^5) + 
          12 \beta^2 (12 + 84 \gamma + 207 \gamma^2 + 240 \gamma^3 + 136 \gamma^4 + 30 \gamma^5) + 
          8 \beta^3 (28 + 168 \gamma + 360 \gamma^2 + 360 \gamma^3 + 171 \gamma^4 + 30 \gamma^5) + 
     5 + 45 \gamma + 144 \gamma^2 +  224 \gamma^3 + 180 \gamma^4 + 60 \gamma^5 + 3 \beta (15 + 120 \gamma + 336 \gamma^2 + 448 \gamma^3 + 300 \gamma^4 + 80 \gamma^5))\Big]
$

\begin{lem} \label{up2} 
One has ${\mathcal  B} < \frac{1}{4}$ throughout the Kahler cone of $\CP\# 3 \overline{\CP}_2$. 
\end{lem}
\begin{proof}
Subtracting four times the numerator from the denominator yields

\bigskip

\noindent 
$
1 + 12 \gamma + 24 \gamma^2 + 8 \gamma^3 - 4 \gamma^4 + 16 \gamma^5 + 48 \gamma^6 + 
  48 \beta^6 (1 + \gamma)^6 + 48\alpha^6 (1 + \beta + \gamma)^6 + 
  16 \beta^5 (1 + 31 \gamma + 93 \gamma^2 + 129 \gamma^3 + 108 \gamma^4 + 60 \gamma^5 + 18 \gamma^6) + 
  4 \beta (3 + 41 \gamma + 116 \gamma^2 + 162 \gamma^3 + 168 \gamma^4 + 124 \gamma^5 + 72 \gamma^6) + 
  8 \beta^2 (3 + 58 \gamma + 162 \gamma^2 + 219 \gamma^3 + 236 \gamma^4 + 186 \gamma^5 + 90 \gamma^6) + 
  8 \beta^3 (1 + 81 \gamma + 219 \gamma^2 + 271 \gamma^3 + 294 \gamma^4 + 258 \gamma^5 + 120 \gamma^6) + 
  4 \beta^4 (-1 + 168 \gamma + 472 \gamma^2 + 588 \gamma^3 + 561 \gamma^4 + 432 \gamma^5 + 
    180 \gamma^6) + 
  16\alpha^5 (1 + 31 \gamma + 93 \gamma^2 + 129 \gamma^3 + 108 \gamma^4 + 60 \gamma^5 + 18 \gamma^6 + 
     18 \beta^6 (1 + \gamma) + 12 \beta^5 (5 + 16 \gamma + 7 \gamma^2) + 
    6 \beta^4 (18 + 102 \gamma + 98 \gamma^2 + 27 \gamma^3) + 
     3 \beta^3 (43 + 324 \gamma + 482 \gamma^2 + 276 \gamma^3 + 54 \gamma^4) + 
    3 \beta^2 (31 + 275 \gamma + 550 \gamma^2 + 482 \gamma^3 + 196 \gamma^4 + 28 \gamma^5) + 
     \beta (31 + 330 \gamma + 825 \gamma^2 + 972 \gamma^3 + 612 \gamma^4 + 192 \gamma^5 + 
       18 \gamma^6)) + 
  4\alpha^4 (-1 + 168 \gamma + 472 \gamma^2 + 588 \gamma^3 + 561 \gamma^4 + 432 \gamma^5 + 
    180 \gamma^6 + 180 \beta^6 (1 + \gamma)^2 + 
     24 \beta^5 (18 + 102 \gamma + 98 \gamma^2 + 27 \gamma^3) + 
     \beta^4 (561 + 6468 \gamma + 9624 \gamma^2 + 5112 \gamma^3 + 936 \gamma^4) + 
     12 \beta^3 (49 + 757 \gamma + 1505 \gamma^2 + 1196 \gamma^3 + 426 \gamma^4 + 54 \gamma^5) + 
     4 \beta (42 + 650 \gamma + 1788 \gamma^2 + 2271 \gamma^3 + 1617 \gamma^4 + 612 \gamma^5 + 
       90 \gamma^6) + 
     2 \beta^2 (236 + 3576 \gamma + 8631 \gamma^2 + 9030 \gamma^3 + 4812 \gamma^4 + 1176 \gamma^5 +      90 \gamma^6)) + 
  4\alpha (3 + 41 \gamma + 116 \gamma^2 + 162 \gamma^3 + 168 \gamma^4 + 124 \gamma^5 + 72 \gamma^6 + 
    72 \beta^6 (1 + \gamma)^5 + 
     4 \beta^5 (31 + 330 \gamma + 825 \gamma^2 + 972 \gamma^3 + 612 \gamma^4 + 192 \gamma^5 + 
       18 \gamma^6) + 
     4 \beta^4 (42 + 650 \gamma + 1788 \gamma^2 + 2271 \gamma^3 + 1617 \gamma^4 + 612 \gamma^5 + 
       90 \gamma^6) + 
     \beta (41 + 570 \gamma + 1812 \gamma^2 + 2816 \gamma^3 + 2600 \gamma^4 + 1320 \gamma^5 + 
       360 \gamma^6) + 
     2 \beta^2 (58 + 906 \gamma + 2832 \gamma^2 + 4189 \gamma^3 + 3576 \gamma^4 + 1650 \gamma^5 + 
       360 \gamma^6) + 
     2 \beta^3 (81 + 1408 \gamma + 4189 \gamma^2 + 5778 \gamma^3 + 4542 \gamma^4 + 1944 \gamma^5 +   360 \gamma^6)) + 
  8\alpha^3 (1 + 81 \gamma + 219 \gamma^2 + 271 \gamma^3 + 294 \gamma^4 + 258 \gamma^5 + 120 \gamma^6 + 
    120 \beta^6 (1 + \gamma)^3 + 
     6 \beta^5 (43 + 324 \gamma + 482 \gamma^2 + 276 \gamma^3 + 54 \gamma^4) + 
    6 \beta^4 (49 + 757 \gamma + 1505 \gamma^2 + 1196 \gamma^3 + 426 \gamma^4 + 54 \gamma^5) + 
     \beta^3 (271 + 5778 \gamma + 14082 \gamma^2 + 14328 \gamma^3 + 7176 \gamma^4 + 1656 \gamma^5 + 
       120 \gamma^6) + 
     \beta (81 + 1408 \gamma + 4189 \gamma^2 + 5778 \gamma^3 + 4542 \gamma^4 + 1944 \gamma^5 + 
       360 \gamma^6) + 
     \beta^2 (219 + 4189 \gamma + 11592 \gamma^2 + 14082 \gamma^3 + 9030 \gamma^4 + 2892 \gamma^5 + 
       360 \gamma^6)) + 
  8\alpha^2 (3 + 58 \gamma + 162 \gamma^2 + 219 \gamma^3 + 236 \gamma^4 + 186 \gamma^5 + 90 \gamma^6 + 
    90 \beta^6 (1 + \gamma)^4 + 
     6 \beta^5 (31 + 275 \gamma + 550 \gamma^2 + 482 \gamma^3 + 196 \gamma^4 + 28 \gamma^5) + 
     \beta^4 (236 + 3576 \gamma + 8631 \gamma^2 + 9030 \gamma^3 + 4812 \gamma^4 + 1176 \gamma^5 + 
       90 \gamma^6) + 
     3 \beta^2 (54 + 944 \gamma + 2838 \gamma^2 + 3864 \gamma^3 + 2877 \gamma^4 + 1100 \gamma^5 + 
       180 \gamma^6) + 
     \beta (58 + 906 \gamma + 2832 \gamma^2 + 4189 \gamma^3 + 3576 \gamma^4 + 1650 \gamma^5 + 
       360 \gamma^6) + 
     \beta^3 (219 + 4189 \gamma + 11592 \gamma^2 + 14082 \gamma^3 + 9030 \gamma^4 + 2892 \gamma^5 + 
       360 \gamma^6))
$

\bigskip

\noindent 
Since this is term-by-term  larger than 
$$4(\alpha^2 -  \alpha^4 +  \alpha^6)+
4(\beta^2 -  \beta^4 +  \beta^6)+
4(\gamma^2 -  \gamma^4 +  \gamma^6) > 0,$$
the denominator is more than  four times larger
than the numerator,  and ${\mathcal B}< \frac{1}{4}$ on the complement
of the hyperplane $\delta=0$ in the K\"ahler cone ${\zap K}$. Since the Futaki invariant
vanishes on this hyperplane \cite{ls}, ${\mathcal B}=0$ there, and we 
therefore have the strict inequality  ${\mathcal B}< \frac{1}{4}$ on all of ${\zap K}$.
       \end{proof}

\begin{lem}\label{pos3}
If $g$ an extremal K\"ahler metric on $M=\CP_2 \# 3 \overline{\CP}_2$, then the scalar curvature
$s$ of $g$ is positive everywhere on $M$.  Moreover, there is a 
continuous function $f : {\zap K}\to \RR$   such that 
$s_{\max} = f (\Omega )$ for any extremal K\"ahler metric, and this $f$ remains bounded 
as one approaches the pull-back of any class from $\CP_2\# 2 \overline{\CP}_2$. 
\end{lem}
\begin{proof}
The  group of permutations of $\alpha$, $\beta$, and $\gamma$ acts
transitively on the  vertices of our hexagon, so it essentially suffices to
compute the value of $s$ at a given vertex, since the maximum and minimum
must occur at some critical point. In fact, evaluating $s$ at the vertex
$(x,y)=(\alpha/2\pi, 0)$ gives 

\bigskip 

\noindent 
$
24\pi \Big[1 + 10 \gamma + 32 \gamma^2 + 48 \gamma^3 + 36 \gamma^4 + 8 \gamma^5 + 8 \beta^5 (1 + \gamma)^4 + 
     8 \alpha^5 (1 + \beta + \gamma)^4 + 
     4 \beta^4 (9 + 44 \gamma + 80 \gamma^2 + 68 \gamma^3 + 25 \gamma^4 + 2 \gamma^5) + 
     8 \beta^3 (6 + 37 \gamma + 80 \gamma^2 + 78 \gamma^3 + 34 \gamma^4 + 4 \gamma^5) + 
     4 \beta^2 (8 + 60 \gamma + 147 \gamma^2 + 160 \gamma^3 + 80 \gamma^4 + 12 \gamma^5) + 
     2 \beta (5 + 44 \gamma + 120 \gamma^2 + 148 \gamma^3 + 88 \gamma^4 + 16 \gamma^5) + 
     4 \alpha^4 (5 + 2 \beta^5 + 24 \gamma + 40 \gamma^2 + 32 \gamma^3 + 13 \gamma^4 + 2 \gamma^5 + 
        \beta^4 (19 + 18 \gamma) + \beta^3 (50 + 96 \gamma + 40 \gamma^2) + 
        2 \beta^2 (29 + 84 \gamma + 72 \gamma^2 + 20 \gamma^3) + 
        2 \beta (15 + 58 \gamma + 75 \gamma^2 + 42 \gamma^3 + 9 \gamma^4)) + 
     8 \alpha^3 (3 + 17 \gamma + 34 \gamma^2 + 35 \gamma^3 + 19 \gamma^4 + 4 \gamma^5 + 
        4 \beta^5 (1 + \gamma) + \beta^4 (25 + 48 \gamma + 20 \gamma^2) + 
        \beta^3 (52 + 151 \gamma + 125 \gamma^2 + 30 \gamma^3) + 
        \beta^2 (52 + 201 \gamma + 246 \gamma^2 + 122 \gamma^3 + 20 \gamma^4) + 
        \beta (23 + 110 \gamma + 177 \gamma^2 + 133 \gamma^3 + 45 \gamma^4 + 4 \gamma^5)) + 
     4 \alpha^2 (4 + 28 \gamma + 69 \gamma^2 + 84 \gamma^3 + 52 \gamma^4 + 12 \gamma^5 + 
        12 \beta^5 (1 + \gamma)^2 + 2 \beta^4 (31 + 90 \gamma + 78 \gamma^2 + 20 \gamma^3) + 
        2 \beta^3 (53 + 210 \gamma + 267 \gamma^2 + 128 \gamma^3 + 20 \gamma^4) + 
        6 \beta^2 (15 + 75 \gamma + 123 \gamma^2 + 86 \gamma^3 + 25 \gamma^4 + 2 \gamma^5) + 
        \beta (35 + 210 \gamma + 420 \gamma^2 + 388 \gamma^3 + 168 \gamma^4 + 24 \gamma^5)) + 
     2 \alpha (3 + 26 \gamma + 74 \gamma^2 + 100 \gamma^3 + 68 \gamma^4 + 16 \gamma^5 + 
        16 \beta^5 (1 + \gamma)^3 + 
        4 \beta^4 (19 + 74 \gamma + 99 \gamma^2 + 54 \gamma^3 + 9 \gamma^4) + 
        4 \beta^3 (28 + 142 \gamma + 243 \gamma^2 + 175 \gamma^3 + 51 \gamma^4 + 4 \gamma^5) + 
        2 \beta^2 (41 + 258 \gamma + 528 \gamma^2 + 470 \gamma^3 + 186 \gamma^4 + 24 \gamma^5) + 
        \beta (28 + 210 \gamma + 498 \gamma^2 + 536 \gamma^3 + 276 \gamma^4 + 
           48 \gamma^5))\Big]\Big/
           \\ 
           \Big[1 + 10 \gamma + 36 \gamma^2 + 64 \gamma^3 + 60 \gamma^4 + 
   24 \gamma^5 + 24 \beta^5 (1 + \gamma)^5 + 24 \alpha^5 (1 + \beta + \gamma)^5 + 
   12 \beta^4 (1 + \gamma)^2 (5 + 20 \gamma + 23 \gamma^2 + 10 \gamma^3) + 
   16 \beta^3 (4 + 28 \gamma + 72 \gamma^2 + 90 \gamma^3 + 57 \gamma^4 + 15 \gamma^5) + 
   12 \beta^2 (3 + 24 \gamma + 69 \gamma^2 + 96 \gamma^3 + 68 \gamma^4 + 20 \gamma^5) + 
   2 \beta (5 + 45 \gamma + 144 \gamma^2 + 224 \gamma^3 + 180 \gamma^4 + 60 \gamma^5) + 
   12 \alpha^4 (1 + \beta + \gamma)^2 (5 + 20 \gamma + 23 \gamma^2 + 10 \gamma^3 + 10 \beta^3 (1 + \gamma) +
       \beta^2 (23 + 46 \gamma + 16 \gamma^2) + 2 \beta (10 + 30 \gamma + 23 \gamma^2 + 5 \gamma^3)) + 
   16 \alpha^3 (4 + 28 \gamma + 72 \gamma^2 + 90 \gamma^3 + 57 \gamma^4 + 15 \gamma^5 + 
      15 \beta^5 (1 + \gamma)^2 + 3 \beta^4 (19 + 57 \gamma + 50 \gamma^2 + 13 \gamma^3) + 
      3 \beta^3 (30 + 120 \gamma + 155 \gamma^2 + 78 \gamma^3 + 13 \gamma^4) + 
      3 \beta^2 (24 + 120 \gamma + 206 \gamma^2 + 155 \gamma^3 + 50 \gamma^4 + 5 \gamma^5) + 
      \beta (28 + 168 \gamma + 360 \gamma^2 + 360 \gamma^3 + 171 \gamma^4 + 30 \gamma^5)) + 
   12 \alpha^2 (3 + 24 \gamma + 69 \gamma^2 + 96 \gamma^3 + 68 \gamma^4 + 20 \gamma^5 + 
      20 \beta^5 (1 + \gamma)^3 + 
      \beta^4 (68 + 272 \gamma + 366 \gamma^2 + 200 \gamma^3 + 36 \gamma^4) + 
      4 \beta^3 (24 + 120 \gamma + 206 \gamma^2 + 155 \gamma^3 + 50 \gamma^4 + 5 \gamma^5) + 
      2 \beta (12 + 84 \gamma + 207 \gamma^2 + 240 \gamma^3 + 136 \gamma^4 + 30 \gamma^5) + 
      \beta^2 (69 + 414 \gamma + 864 \gamma^2 + 824 \gamma^3 + 366 \gamma^4 + 60 \gamma^5)) + 
   2 \alpha (5 + 45 \gamma + 144 \gamma^2 + 224 \gamma^3 + 180 \gamma^4 + 60 \gamma^5 + 
      60 \beta^5 (1 + \gamma)^4 + 
      12 \beta^4 (15 + 75 \gamma + 136 \gamma^2 + 114 \gamma^3 + 43 \gamma^4 + 5 \gamma^5) + 
      12 \beta^2 (12 + 84 \gamma + 207 \gamma^2 + 240 \gamma^3 + 136 \gamma^4 + 30 \gamma^5) + 
      8 \beta^3 (28 + 168 \gamma + 360 \gamma^2 + 360 \gamma^3 + 171 \gamma^4 + 30 \gamma^5) + 
      3 \beta (15 + 120 \gamma + 336 \gamma^2 + 448 \gamma^3 + 300 \gamma^4 + 80 \gamma^5))\Big]$
      
      \bigskip
      
      \noindent
which is  smooth and term-by-term positive  for $\alpha, \beta, \gamma \geq 0$.

This expression can be uniquely extended to all $\delta > 0$ by turning
the numerator  and denominator into  homogeneous polynomials of 
$(\alpha, \beta , \gamma, \delta)$
 of degree $9$ and $10$, respectively. 
 The resulting expression is then smooth across $\delta =0$,
 because the  numerator and denominator of the above expression actually do  contain 
 some terms of degree $9$ and $10$, respectively. 
  Permuting $\alpha$, $\beta$, and  $\gamma$,
we obtain six smooth positive functions. Taking the minimum of these then 
shows that $s_{\min}$ is everywhere positive, while taking the maximum
produces the required continuous positive function $f: {\zap K}\to \RR$.
\end{proof}

\pagebreak

\vfill 

\noindent 
{\bf Acknowledgements.} The author would like to warmly   thank
 Xiuxiong Chen and Yuan Fang for their critical reading of an early draft of 
 the manuscript, which resulted in  substantial corrections, and eventually led to   major
 simplifications  and improvements.
 
 \bigskip 

\noindent
{\sc Author's address:} 

\medskip 

 \noindent
{Mathematics Department, SUNY, Stony Brook, NY 11794, USA
}

\bigskip

\noindent
{\sc Author's e-mail:} 

\medskip 

 \noindent
{\tt claude@math.sunysb.edu
}

  \end{document}